\documentclass[12pt]{amsart}
\usepackage{amssymb, amsmath, amsthm, epsf, epsfig, color, bbold}

\newlength{\originalbase}
\begin{document}
\newcounter{thmenumerate}
\newenvironment{thmenumerate}
{\setcounter{thmenumerate}{0}%
 \renewcommand{\thethmenumerate}{\textup{(\roman{thmenumerate})}}%
 \def\item{\par
 \refstepcounter{thmenumerate}\textup{(\roman{thmenumerate})\enspace}}
}

\newtheorem{theorem}{Theorem}[section]
\newtheorem{claim}{Claim}[theorem]
\newtheorem{prop}[theorem]{Proposition}
\newtheorem{remark}[theorem]{Remark}
\newtheorem{lemma}[theorem]{Lemma}
\newtheorem{corollary}[theorem]{Corollary}
\newtheorem{conjecture}[theorem]{Conjecture}
\newtheorem{example}[theorem]{Example}
\newcommand{\restr}{{\upharpoonright}}
\newcommand{\sm}{{\setminus}}
\newcommand{\E}{{\mathbb E}}
\newcommand{\N}{{\mathbb N}}
\newcommand{\R}{{\mathbb R}}
\newcommand{\Z}{{\mathbb Z}}
\newcommand{\bone}{{\mathbb 1}}
\newcommand{\cA}{{\mathcal A}}
\newcommand{\cB}{{\mathcal B}}
\newcommand{\cC}{{\mathcal C}}
\newcommand{\cE}{{\mathcal E}}
\newcommand{\cF}{{\mathcal F}}
\newcommand{\cG}{{\mathcal G}}
\newcommand{\cH}{{\mathcal H}}
\newcommand{\cK}{{\mathcal K}}
\newcommand{\cM}{{\mathcal M}}
\newcommand{\cP}{{\mathcal P}}
\newcommand{\cT}{{\mathcal T}}
\newcommand{\cW}{{\mathcal W}}
\newcommand{\cZ}{{\mathcal Z}}
\newcommand{\Prob}{{\mathbb P}}
\newcommand{\eps}{{\varepsilon}}
\renewcommand\P{\operatorname{\mathbb P{}}}

\title{Order-invariant Measures on Fixed Causal Sets}

\author{Graham Brightwell}
\address{Department of
Mathematics, London School of Economics and Political Science,
Houghton Street, London WC2A 2AE}
\email{g.r.brightwell@lse.ac.uk}

\author{Malwina Luczak}
\address{School of Mathematics and Statistics, University of Sheffield,
Hicks Building, Hounsfield Road, Sheffield S3 7RH}
\email{m.luczak@sheffield.ac.uk}

\thanks{The research of Malwina Luczak was carried out at the London School
of Economics, and was supported in part by a grant from STICERD}


\begin{abstract}
A causal set is a countably infinite poset in which every element is above
finitely many others; causal sets are exactly the posets that have a
linear extension with the order-type of the natural numbers -- we call
such a linear extension a {\em natural extension}.   We study probability
measures on the set of natural extensions of a causal set, especially
those measures having the property of {\em order-invariance}: if we
condition on the set of the bottom $k$ elements of the natural extension,
each feasible ordering among these $k$ elements is equally likely.  We
give sufficient conditions for the existence and uniqueness of an
order-invariant measure on the set of natural extensions of a causal set.
\end{abstract}

\maketitle

\section {Introduction}

For a finite partially ordered set (poset) $P=(X,<)$, a
{\em linear extension} of $P$ is a linear order on $X$ extending the
partial order $<$.  The notion of a uniform random linear extension of $P$
arises in a number of contexts, see for instance~\cite{Bri3,Wink},
enabling meaning to be given to the probability that $x$ is below $y$,
when $x$ and $y$ are incomparable.

We pick out one property possessed by the uniform measure in the finite
case.  A {\em down-set} in a poset $P=(X,<)$ is a subset $D$ of $X$ such 
that, if $x\in D$ and $y<x$, then $y \in D$.  For $A$ a down-set in $P$ of 
size $k$, if we consider any linear extension of $P$ in which the bottom 
$k$ elements are the elements of $A$, then the order on these elements is a 
linear extension of the poset $P_A$ induced by $P$ on $A$.  It is easy to 
see that, under the uniform probability measure, if we condition on the 
event that the bottom $k$ elements are those in $A$, then each linear 
extension of $P_A$ is equally likely.

Our aim in this paper is to initiate study of the case where $P$ is
countably infinite, imposing the property above -- which we shall call
{\em order-invariance} -- as an axiom.  This condition, enabling a passage
from the finite to the infinite, is hopefully reminiscent of the notion of
a {\em Gibbs measure} from statistical physics.

As we shall see, depending on $P$, there may be one, many, or no
order-invariant probability measures on the set of linear extensions of
$P$ (or ``on $P$'', for short).  Our results include sufficient conditions
for the existence of an order-invariant measure on $P$, and sufficient
conditions for uniqueness.  We also give a number of examples, including
one class of posets -- the downward-branching trees $T$ -- for which we
give a surprisingly subtle answer to the question of when there is
an order-invariant measure on $T$.

Our need to be able to discuss the ``bottom $k$ elements'' in a linear
extension of $P$ leads us to restrict the class of countable posets we
deal with, and also the class of their linear extensions.

A {\em causal set} is a countably infinite partially ordered set $P=(Z,<)$
such that every element is above only finitely many others.  A causal set
is exactly a poset that has a linear extension with the order-type of $\N$,
i.e., a bijection $\lambda: \N \to Z$ such that we never have $i<j$ and
$\lambda(i) > \lambda(j)$.  We call such a linear extension of a countable
poset a {\em natural extension}.

A probability measure on the set of natural extensions of a causal set $P$
is {\em order-invariant} if, for each $k \in \N$ and each $k$-element
down-set $A$ of $P$, conditioned on the event that
$\{\lambda(1), \dots, \lambda(k)\} = A$, each linear extension of $P_A$ is
equally likely to be the restriction of $\lambda$ to $[k]$.  

In this paper, all our measures will be probability measures, although we 
often omit explicit mention of this; for instance, we will write 
``order-invariant measure'' instead of ``order-invariant probability measure''. 

We give a simple example, to illustrate the definitions and to show that
there are posets $P$ with more than one order-invariant measure on $P$.

\bigbreak

\noindent{\bf Example 1}. \quad 
Let $P$ be the causal set made up of the disjoint union of two infinite
chains $B: b_1<b_2<\cdots$ and $C:c_1<c_2<\cdots$.  Not every linear
extension of $P$ is a natural extension: for instance
$b_1<b_2< \cdots < c_1<c_2< \cdots$ is a linear extension that does not
have the order-type of $\N$.

We shall consider natural extensions of $P$ as constructed ``from the
bottom up''.  At each stage, after we have selected the lowest $k$
elements $x_1,x_2,\dots,x_k$ of the linear extension, the next element
$x_{k+1}$ must be a minimal element among those not yet selected, and
there will always be exactly two candidates, one in $B$ and one in $C$.
To prescribe how to generate a ``random linear extension'' of $P$, we need
to give a probabilistic rule stating how to choose between these two
elements.

Given a parameter $q \in [0,1]$, one such rule is ``always choose the
minimal remaining element of $B$ with probability $q$, and the minimal
remaining element of $C$ with probability $1-q$''.  This rule gives us a
probability measure $\mu_q$ on the set of natural extensions of $P$
(equipped with a $\sigma$-field that we shall specify later).

To see that $\mu_q$ is order-invariant, consider any $k$-element down-set
$A$ of $P$, so $A = \{b_1, \dots, b_\ell, c_1, \dots, c_{k-\ell}\}$ for
some $\ell$.  For any of the $\binom{k}{\ell}$ linear extensions
$a_1<\cdots< a_k$ of $P_A$, the a priori probability that the random
linear extension ``starts'' $a_1<\cdots<a_k$ is equal to
$q^\ell (1-q)^{k-\ell}$.  Thus, conditioned on the bottom $k$ elements
being the elements of $A$, each of the $\binom{k}{\ell}$ linear
extensions of $P_A$ is equally likely to be the order among the elements
of $A$.

Thus we have an uncountable family of order-invariant measures on $P$.

An order-invariant measure on $P$ is said to be {\em extremal} if it
cannot be expressed as a convex combination of two other order-invariant
measures on $P$.  We shall return to this example later and show that the
$\mu_q$ are the only extremal order-invariant measures on $P$.  All other
order-invariant measures can be constructed according to a two-stage rule:
first choose $q$ according to some probability distribution on $[0,1]$,
then choose the linear extension according to $\mu_q$.

\bigbreak

This work is part of a wider project, initiated in our companion
paper~\cite{BL1}.  In that paper, we consider probability measures where
the causal set $P$ is also random.  More precisely, we consider processes
that generate a causal set one element at a time, at each stage adding a
maximal element, with a label drawn from a given set (which we take to be
the interval $[0,1]$), and putting the new element above some down-set in
the current poset.  Such processes are called {\em causal set processes}:
formally they are Markov processes, whose states are pairs
$(x_1\cdots x_k, <^{[k]})$, where $x_1\cdots x_k$ is a string of elements
from $[0,1]$, and $<^{[k]}$ is a partial order on the index set $[k]$
that is a suborder of the natural order on $[k]$.  Each state corresponds
to a partial order $P_k$ on the set $X_k = \{x_1, \dots, x_k\}$ -- given
by $x_i < x_j$ if and only if $i <^{[k]} j$ -- together with a linear
extension of $P_k$.

Let us indicate, fairly precisely, how a probability measure on a fixed
causal set $P=(Z,<)$, with $Z \subset [0,1]$, fits into this framework.
Consider a causal set process where the only allowed transitions are to
states $(x_1\cdots x_k, <^{[k]})$, where $X_k= \{x_1, \dots, x_k\}$ is a
finite down-set in $P$, and $i <^{[k]} j$ if and only if $x_i<x_j$.  In
other words, the derived poset $P_k$ is the restriction of $P$ to $X_k$.
Effectively, a transition always adds a minimal element $x_{k+1}$ of
$P\sm X_k$ to the end of the string $x_1\cdots x_k$, and augments the
poset $<^{[k]}$ according to which elements of $X_k$ are below $x_{k+1}$
in $P$.  In such a process, the order $<^{[k]}$ can be derived from the
string $x_1\cdots x_k$ and the causal set $P$, and so it can be omitted
from the notation.  A sample path of the process gives rise to an
infinite string $x_1x_2\cdots$ of elements of $Z$: if it happens that
$X = \{ x_1,x_2,\dots \} = Z$, then this will be a natural extension
of~$P$.

A consequence of the main result of~\cite{BL1} is that, to classify the
extremal order-invariant measures in this broader setting, it is enough
to classify the extremal order-invariant measures on fixed causal sets.
However, that is likely to be a prohibitively difficult task: giving
conditions for existence and/or uniqueness of order-invariant measures on
a fixed $P$ is a more realistic goal.


Besides the inherent interest, another motivation for studying
order-invariant measures comes from physics, in the context of a proposal
for a random causal set as a mathematical model of space-time.  Rideout
and Sorkin~\cite{RS} gave various desirable conditions for
such a model, including order-invariance.  Although the proposed list of
conditions turns out to be too narrow to include causal sets resembling
the observed space-time universe (see~\cite{BG}), we are led to ask
whether order-invariance itself is an obstacle: we return to this in the
Open Problems at the end of the paper.

We mention some other connections with earlier work.

Some years ago, the first author~\cite{Bri1,Bri2} studied random linear
extensions of locally finite posets.  The main theorem of~\cite{Bri1},
interpreted in the present context, is as follows.  If a causal set $P$
has the property that, for some fixed $k$, every element is incomparable
with at most $k$ others, then there is a unique order-invariant measure on
$P$.  The interpretation is spelled out in Theorem~\ref{thm:unique} of the
present paper.

The specific case where the causal set is the two-dimensional grid
$G= (\N \times \N,<)$ has attracted attention from another direction, as
it is connected with the representation theory of the infinite symmetric
group, and with harmonic functions on the Young lattice (which is the
lattice of down-sets of $G$).  A good account of this theory appears in
Kerov~\cite{Kerov}, where a somewhat more general theory is also
developed.  Our concerns in this paper are rather different, but the two
theories have various points of contact.

The case of order-invariant measures on a fixed causal set $P$ can also be
viewed as a (1-dimensional) spin system.  There are (at least) two ways to
do this: either we can treat the elements as particles, with the spin of
an element $z$ encoding its rank $\lambda^{-1}(z)$ in a natural extension
$\lambda$, or we can treat the pairs of incomparable elements as
particles, with the spin of a pair determining which is higher in the
natural extension.  Thus some of the general results discussed in, for
instance, Bovier~\cite{Bovier} or Georgii~\cite{Georgii} apply.  (Indeed,
some of the results in~\cite{Georgii} hold also for general
order-invariant measures, as is explained in~\cite{BL1}.)

The structure of the paper is as follows.  Basic definitions and notation
connected with causal sets and natural extensions are given in
Section~\ref{sec:csne}.  In Section~\ref{sec:pfcs}, we give a full
specification of the probability spaces we work in, and of the notion of
order-invariance.  Section~\ref{sec:two-ex} is devoted to a simple example
worked out in some detail.  In Section~\ref{sec:extremal}, we state a
consequence of a result from~\cite{BL1}, giving different
characterisations of extremal order-invariant measures.

Our formal definition of order-invariance includes processes that are not
natural extensions of $P$, but instead are natural extensions of the
restriction $P_Y$ of $P$ to some infinite down-set in $P$.  An
order-invariant measure that does a.s.\ give a natural extension of $P$ is
called {\em faithful}, and we investigate this concept in
Section~\ref{sec:fnfp}.

As we have mentioned, we are particularly interested in the following two
questions.  For which causal sets $P$ is there an order-invariant measure
on $P$?  For which causal sets $P$ is there a unique order-invariant
measure on $P$?  In Section~\ref{sec:existence}, we show that any causal
set $P$ with no infinite antichain admits an order-invariant measure.  In
Section~\ref{sec:uoim}, we show that, for any causal set $P$ where there
is a uniform bound $k$ on the number of elements incomparable with an
element $x$, there is just one order-invariant measure on $P$.  As
mentioned above, this is a simple application of the main result
of~\cite{Bri1}.

These conditions for existence and uniqueness are far from necessary, and
in particular it seems that any description of which causal sets
admit an order-invariant measure must be significantly more complicated.
In Section~\ref{sec:dbt}, we show that a downward-branching tree $T$
admits an order-invariant measure if and only if a certain series of
numbers derived from $T$ is convergent.

In Section~\ref{sec:grid}, we briefly discuss the case of the
two-dimensional grid poset studied by Kerov~\cite{Kerov} and others.

One question that we have not answered is the one that originally
motivated this research: is there an order-invariant process that gives
rise to causal sets resembling discrete approximations to the space-time
structure of the universe?  This and other open problems are discussed in
Section~\ref{sec:op}.

\section {Causal Sets and Natural Extensions} \label{sec:csne}

A {\em (labelled) poset} $P$ is a pair $(Z,<)$, where $Z$ is a set (for
us, $Z$ will always be countable), and $<$ is a partial order on $Z$.  A
{\em total order} or {\em linear order} on $Z$ is a poset such that each
pair of elements of $Z$ is comparable.

A {\em down-set} in $P$ is a subset $Y \subseteq Z$ such that, if $a\in Y$
and $b<a$, then $b \in Y$.  A {\em stem} is a finite down-set (this term
is less standard: it has been used in some physics papers).  An
{\em up-set} is the complement of a down-set.

If $P=(Z,<)$ is a poset, and $Y \subseteq Z$, then $<_Y$ denotes the
restriction of the partial order to $Y$, and $P_Y=(Y,<_Y)$.  For
$W \subset Z$, we also write $P\sm W$ to mean $P_{Z\sm W}$.

A pair $(x,y)$ of elements of $Z$ is a {\em covering pair} if $x<y$, and
there is no $z\in Z$ with $x<z<y$.

For a poset $P=(Z,<)$ and an element $x\in Z$, set
$D(x)=\{ y \in Z: y<x\}$, $U(x) = \{ y \in Z : y > x\}$ and let $I(x)$ be
the set of elements incomparable with $x$.  We also define
$D[x] = D(x) \cup \{ x\}$ and $U[x] = U(x) \cup \{ x\}$.

Let $P=(Z,<)$ be a poset on a countably infinite set $Z$.  We say that $P$
is a {\em causal set} (or {\em causet}) if $D(z)$ is finite for each
$z \in Z$.

A {\em linear extension} of a poset $P=(Z,<)$ is a total order $\prec$ on
$Z$ such that, whenever $x<y$, we also have $x \prec y$.

The sets $\N$ and $[k]=\{1,\dots, k\}$, for $k \in \N$, come equipped with 
a ``standard'' linear order.  In these cases, a {\em suborder} of $\N$ or
$[k]$ will be a partial order on that ground-set (typically denoted $<^\N$
or $<^{[k]}$) with the standard order as a linear extension, i.e., if
$<^\N$ is a suborder of $\N$ and $i<^\N j$, then $i$ is below $j$ in the
standard order on $\N$.

A {\em natural extension} of a causal set $P=(Z,<)$
is a bijection $\lambda$ from $\N$ to $Z$ such that
$\lambda^{-1}$ is order-preserving: i.e., if $\lambda(i) < \lambda(j)$,
then $i<j$.  We shall often write natural extensions as $x_1x_2 \cdots$,
meaning that $\lambda(i) =x_i$.  In this notation, an
{\em initial segment} of $\lambda$ is an initial substring
$x_1x_2 \cdots x_k$, for some $k \in \N$.

A natural extension $\lambda$ of $P=(Z,<)$ gives rise to a linear
extension $\prec$ by setting $x \prec y$ whenever
$\lambda^{-1}(x) < \lambda^{-1}(y)$.  The linear extensions arising in
this way are those with the order-type of $\N$.

Similarly, if $P=(Z,<)$ is a finite poset, with $|Z|=k$, we can think of a
linear extension as a bijection $\lambda: [k] \to Z$ such that
$\lambda^{-1}$ is order-preserving, i.e., if $\lambda(i) < \lambda(j)$,
then $i<j$ in~$[k]$.  We shall sometimes write a linear extension of
$P$ as a string $x_1\cdots x_k$, meaning that $\lambda(i) = x_i$ for
$i=1,\dots,k$: in this sense, we can again talk of an initial segment
of a linear extension.  For finite partial orders, we shall use these
various equivalent notions of linear extension interchangeably.  For a
finite poset $P$, let $e(P)$ denote the number of linear extensions
of~$P$.

An {\em ordered stem} of a causal set, or a finite poset, $P=(Z,<)$, is
a finite string $x_1\cdots x_k$ such that $X=\{x_1, \dots, x_k\}$ is a
down-set in $P$, and $x_1\cdots x_k$ is a linear extension of $P_X$.
Ordered stems of a causal set (finite poset) $P$ are exactly the strings
that can arise as an initial segment of a natural (linear) extension
of~$P$.

For a causal set or finite poset $P$, and an ordered stem $x_1\cdots x_k$
of $P$, let $E^P(x_1\cdots x_k)$ denote the set of natural/linear
extensions of $P$ with initial segment $x_1\cdots x_k$.  When there is
only one poset $P$ under consideration, we shall use the simpler notation
$E(x_1 \cdots x_k)$ instead.

For a causal set $P$, let $L(P)$ denote the set of natural extensions
of $P$.  Also, let $L'(P)$ denote the set of injections $\lambda$ from
$\N$ to $P$ such that, for each $i$,
$D(\lambda(i)) \subseteq \{\lambda(1), \dots, \lambda(i-1)\}$.  In
general, elements of $L'(P)$ need not be bijections: those elements of
$L'(P)$ that are bijections are exactly the natural extensions.

The following statements are all very straightforward to verify.
A countable poset has a natural extension if and only if every element is
above finitely many elements, i.e., if and only if it is a causal set. If
a causal set $P$ has no element $x$ with $I(x)$ infinite, then all linear
extensions of $P$ are natural extensions, and $L(P)=L'(P)$.  However, if
there is an element $x$ of $P$ with $I(x)$ infinite, then there is
(a)~a linear extension of $P$ that does not have the order-type of $\N$
and (b)~an element of $L'(P)$ whose image is the proper subset
$I(x) \cup D(x)$ of~$P$.

\section {Order-invariant Processes on Fixed Causal Sets} \label{sec:pfcs}

Consider a fixed causal set $P=(Z,<)$, with $Z$ a countable subset of
$[0,1]$.  (The actual nature of the set $Z$ is not crucial; we demand that
the labels of our posets are taken from $[0,1]$ only in order to
incorporate the structures studied in this paper within the general
framework of~\cite{BL1}.)

For $k$ a non-negative integer, let $\cE_P^{[k]}$ denote the set of
ordered stems of $P$ with $k$ elements.  Let $\cE_P$ be the union of the
$\cE_P^{[k]}$, i.e., the set of all ordered stems of $P$.

A {\em causet process on $P$} is a discrete-time Markov chain with state
space $\cE_P$, such that the only allowed transitions from a state
$x_1\cdots x_k \in \cE_P^{[k]}$ are those to a state
$x_1\cdots x_kx_{k+1} \in \cE_P^{[k\!+\!1]}$, where $x_{k+1}$ is a minimal
element of $P \sm X_k$, where $X_k = \{ x_1, \dots, x_k\}$.

Sample paths of a causet process on $P$, starting from the empty ordered
stem, correspond to natural extensions $x_1x_2 \cdots$ of some restriction
$P_X$ to an infinite down-set $X = \{ x_1,x_2,\dots\}$ of $P$.  Indeed,
given a natural extension $x_1x_2 \cdots$, its finite initial segments
form a possible sample path of a causet process on $P$.  It is thus
natural to work with a sample space whose elements are these natural
extensions.

Accordingly, for a causal set $P$, we define $\Omega^P$ to be the set of
infinite strings $\omega = x_1x_2\cdots$ that are natural extensions of
$P_X$ for some down-set $X = \{ x_1, x_2, \dots\}$ in $P$.  Equivalently,
$\Omega^P$ is the set of strings $\omega = x_1x_2\cdots$ such that, for
each $k\in \N$, $x_k$ is a minimal element of
$P\sm X_{k-1}$, where $X_{k-1} = \{ x_1, \dots, x_{k-1}\}$.

For $a_1a_2 \cdots a_k$ an ordered stem of $P$, we define
$E(a_1\cdots a_k) = E^P(a_1\cdots a_k)$ to be the set of elements of
$\Omega^P$ with $a_1\cdots a_k$ as an initial segment.  In other words,
$$
E(a_1\cdots a_k) = \{ \omega=x_1x_2\cdots \in \Omega^P :
x_1=a_1,\dots,x_k=a_k\}.
$$
A set of this form is called a {\em basic event} (for $P$).

For fixed $k$, let $\cF^P_k$ be the $\sigma$-field generated by the
events $E(a_1\cdots a_k)$, for $a_1\cdots a_k$ an ordered stem of
length $k$.  Also, let $\cF^P$ be the $\sigma$-field generated by the
union of the $\cF^P_k$.

A {\em causet measure} on $P$ is a probability measure on $(\Omega^P,\cF^P)$.

A {\em separating class} in $(\Omega^P,\cF^P)$ is a subset $\cH$ of
$\cF^P$ such that, if two probability measures agree on $\cH$, then they
are equal.  For any causal set $P$, the collection of basic events
$E(a_1\cdots a_k)$, for $a_1\cdots a_k$ an ordered stem of $P$, forms a
separating class.

The sequence $(\cF_k^P)$ is the natural filtration for a causet process on
$P$.  The measure $\mu$ of a causet process on $P$ is determined by the
finite-dimensional distributions of the Markov process, i.e., by its
values on the sets $E(a_1\cdots a_k)$.

We can equip $\Omega^P$ with a metric in several natural ways, many of
which lead to equivalent topologies.  For instance we can define the
metric by
$$
d(x_1x_2\cdots, y_1y_2\cdots) =
\sum_{i=1}^\infty 2^{-i} \bone(x_i \not= y_i).
$$

\begin{theorem} \label{thm:compact}
Let $P=(Z,<)$ be a causal set.  The space $\Omega^P$, with the metric
above, is compact if and only if, for all stems $A$ of $P$, $P\sm A$ has
finitely many minimal elements.
\end{theorem}

If $P$ has no infinite antichain, then the condition above is satisfied,
since the set of minimal elements of $P\sm A$, for any stem $A$, is an
antichain.  However, the condition in the theorem is weaker: consider a
chain $a_1<a_2< \cdots$, with incomparable infinite chains placed above
each $a_i$.  This poset has an infinite antichain, but deleting any stem
leaves a causal set with finitely many minimal elements.

\begin{proof}
Suppose first that, for each stem $A$ of $P$, $P\sm A$ has finitely many
minimal elements.  Consider any sequence $(\omega^m)$ of elements of
$\Omega^P$.  We show that there is a convergent subsequence
$(\omega^{m_j})$ of $(\omega^m)$.  The argument is very standard.

We construct an element $\omega^0 = a_1a_2\cdots $ of $\Omega^P$ with the
property that, for each $j \in \N$, the ordered stem $a_1\cdots a_j$ is an
initial segment of infinitely many of the $\omega^m$.  Once we have done
this, the result follows: for each $j$ in turn, we choose $m_j > m_{j-1}$
so that $\omega^{m_j}$ has $a_1\dots a_j$ as an initial segment -- now the
subsequence $(\omega^{m_j})$ converges to $\omega^0$.

We construct $\omega^0$ recursively.  For $j \ge 0$, suppose that
$a_1 \cdots a_j$ is an ordered stem in $P$ that is an initial segment of
an infinite set $\{ \omega^{m_1},\omega^{m_2}, \dots\}$ of the elements
$\omega^m$.  Now the set $B$ of minimal elements of
$P \sm \{ a_1, \dots, a_j\}$ is finite.  Moreover, the next entry of each
of the $\omega^{m_i}$ is an element of $B$, so some element $a_{j+1}$ of
$B$ occurs infinitely often as the next element in $\omega^{m_i}$, and
hence the ordered stem $a_1\cdots a_ja_{j+1}$ occurs infinitely often as
an initial segment.  Proceeding in this way, we may construct a suitable
$\omega^0$.

Conversely, suppose that there is a stem $A$ of $P$ such that the set $M$
of minimal elements of $P\sm A$ is infinite.  We take some enumeration
$b_1,b_2,\cdots $ of $M$, and any linear extension $a_1\cdots a_k$ of
$P_A$, and define $\omega^i = a_1\cdots a_kb_ib_{i+1}b_{i+2} \cdots$, for
$i \in \N$.  We see that each string $\omega^i$ is in $\Omega^P$, and that
$d(\omega^i, \omega^j) = \sum_{\ell = k+1}^\infty 2^{-\ell} = 2^{-k}$
whenever $i\not= j$.  Therefore the sequence $(\omega^i)$ of elements of
$\Omega^P$ does not have a convergent subsequence, and so the space
$\Omega^P$ is not compact.
\end{proof}

We need some notation for functions on $\Omega^P$, i.e., random elements on
our probability space.  If $\omega = x_1x_2\cdots$, then we set
$\xi_j(\omega) = x_j$, $\Xi_j(\omega) = \{ x_1, \dots, x_j\}$, and
$\Xi(\omega) = \{x_1, x_2, \dots\}$.

We say that a causet measure $\mu$ on $P$ is {\em order-invariant} if,
whenever $A = \{ a_1, \dots, a_k\}$ is a stem of $P$, and $s$ is a
permutation of $[k]$ such that both $a_1a_2\cdots a_k$ and
$a_{s(1)}a_{s(2)}\cdots a_{s(k)}$ are linear extensions of $P_A$, then
\begin{equation} \label{eq:o-i}
\mu(E(a_1\cdots a_k)) = \mu(E(a_{s(1)}\cdots a_{s(k)})).
\end{equation}
We say that a causet process on a causal set $P$ is order-invariant if the
corresponding causet measure is order-invariant.

We can rephrase the condition of order-invariance in several different
ways.

For $A = \{ a_1, \dots, a_k\}$ a stem of $P$, let $\nu^A$ denote the
uniform measure on linear extensions of the finite poset $P_A$.  Then the
causet measure $\mu$ on $P$ is order-invariant if and only if, for every
stem $A = \{a_1, \dots, a_k\}$ of $P$, and every linear extension
$a_1\cdots a_k$ of $P_A$,
\begin{equation}\label{uniform}
\mu(E^P(a_1\cdots a_k)\mid \Xi_k = A) = \nu^A(\{a_1\cdots a_k\}) =
\frac{1}{e(P_A)}.
\end{equation}

More generally, if $\mu$ is an order-invariant measure on $P$, $A$ is a
stem of $P$ of size $k$, $\ell$ is a natural number with $\ell \le k$, 
and $a_1\cdots a_\ell$ is an ordered stem whose elements are all in $A$, 
then
\begin{equation} \label{uniform2}
\mu(E^P(a_1\cdots a_\ell) \mid \Xi_k = A) =
\nu^A(E^{P_A}(a_1\cdots a_\ell)).
\end{equation}
This identity is obtained by summing (\ref{uniform}) over the elements
of $E^{P_A}(a_1\cdots a_\ell)$.

There is a strong similarity between order-invariance and the
{\em Gibbs measure} condition from statistical physics: if we take any
finite patch of a space, and condition on the configuration outside that
patch (here, that means conditioning on the event that the set $\Xi_k$ of
the first $k$ elements -- i.e., those not accounted for outside the patch
-- is equal to a given set $A$), then all legal extensions of the
configuration into the patch (here, all linear extensions of the order
restricted to $A$) are equally likely (or, more generally, have some
specified relative probabilities).  See Georgii~\cite{Georgii} or
Bovier~\cite{Bovier} for a very general treatment of Gibbs measures.

To check order-invariance, it is enough to verify condition~(\ref{eq:o-i})
above when $s$ is an adjacent transposition, and the two transposed
elements are incomparable.  This is an easy consequence of the fact
that it is possible to step between any two linear extensions of a finite
poset by exchanges of adjacent incomparable elements.

A causet process on $P$ is {\em order-Markov} if the transition
probabilities out of a state $x_1\cdots x_k \in \cE_P$ depend only on the
set $X_k = \{x_1\cdots x_k\}$, and not on the order of the elements.
A causet measure $\mu$ on $P$ is order-Markov if its associated process is:
this means that
$$
\frac{\mu(E(a_1\cdots a_kb))}{\mu(E(a_1\cdots a_k))} =
\frac{\mu(E(a_{s(1)}\cdots a_{s(k)}b))}{\mu(E(a_{s(1)}\cdots a_{s(k)}))},
$$
whenever $a_1\cdots a_k$ and $a_{s(1)}\cdots a_{s(k)}$ are ordered stems
of $P$, $s$ is a permutation of $[k]$, $\mu(E(a_1\cdots a_k)) > 0$, and
$b$ is a minimal element of $P\sm \{a_1,\dots,a_k\}$.

If $\mu$ is an order-invariant measure on $P$, then it is also
order-Markov, as the numerators and denominators above are equal.  The
converse is far from true: as an extreme example, consider a causet measure
$\mu_{x_1x_2\cdots}$ on a causal set $P$ where the probability of one
specified natural extension $x_1x_2\cdots$ of $P$ is~1: this measure
$\mu_{x_1x_2\cdots}$ is trivially order-Markov, but not order-invariant
unless $x_1x_2 \cdots$ forms a chain.

However, if we know that a causet measure $\mu$ arises from an
order-Markov process, then in order to check order-invariance, it is
enough to verify that~(\ref{eq:o-i}) holds when $s$ is the permutation
exchanging the {\em last} two incomparable elements: if this holds, then
the order-Markov condition implies that~(\ref{eq:o-i}) holds whenever $s$
is an exchange of any pair of incomparable elements, and we have already
remarked that this suffices for order-invariance.  We shall make use of
this later.

We next give an easy but useful lemma, telling us what conditions need to
be checked to ensure that a given specification of values
$\mu(E(a_1\cdots a_k))$ defines a measure on $(\Omega^P, \cF^P)$, for a
given causal set~$P$.

\begin{lemma} \label{to-check}
Let $P=(Z,<)$ be a causal set, and let $f$ be a function from the set of
ordered stems of $P$ to $[0,1]$.  Setting
$\mu(E(a_1\cdots a_k)) = f(E(a_1\cdots a_k))$ defines a measure on
$(\Omega^P,\cF^P)$ if and only if the following hold:
\begin{itemize}
\item[(i)] $f(\phi) = 1$, where $\phi$ denotes the empty string,
\item[(ii)] for each ordered stem $a_1\cdots a_k$, we have
$$
\sum _b f(a_1\cdots a_kb) = f(a_1\cdots a_k),
$$
where the sum runs over all minimal elements $b$ of
$P\sm \{a_1,\dots, a_k\}$.
\end{itemize}
\end{lemma}

The conditions of the lemma amount to Kolmogorov's consistency
conditions; see Chapter~8 in~\cite{GrSt}.
The proof is routine and omitted.

Thus, to check that $\mu$ is an order-invariant measure on a given causal
set $P$, we need to check~(i) (which is usually trivial) and (ii), and
also the order-invariance condition.

\section{An Example} \label{sec:two-ex}

In this section, we study one specific example in detail, both to illustrate
the definitions and themes of the paper and to provide an explicit (non-trivial)
example of a causal set $P$ such that there is exactly one order-invariant measure
on $P$.

\bigbreak

\noindent {\bf Example 2}. \quad
Figure~\ref{fig:ladder} below shows the Hasse diagram of a labelled causal
set $P=(Z,<)$, where $Z = \{ b_1, b_2, \dots\}$, and $b_j>b_i$ if $j>i+1$.

\begin{figure} [hbtp]
\epsfxsize140pt
$$\epsfbox{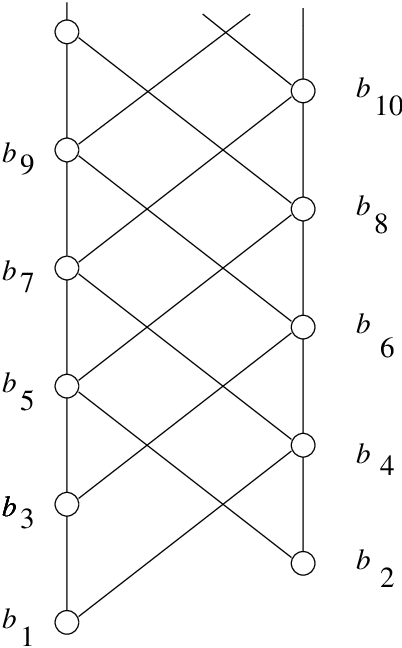}$$
\caption{The causal set $P=(Z,<)$} \label{fig:ladder}
\end{figure}

We will show, in some detail, that there is exactly one order-invariant
measure on $P$.  Some of the methods we use to study this example will be
seen in more generality later.

For $n\in \N$, set $Z_n = \{b_1, \dots, b_n\}$, and $P_n = P_{Z_n}$, the
restriction of $P$ to $Z_n$.  The linear extensions of $P_n$ either have
$b_n$ as the top element, or have $b_{n-1}$ top and $b_n$ next top.  The
former set of linear extensions is in 1-1 correspondence with the set of
linear extensions of $P_{n-1}$, and the latter set is in 1-1
correspondence with the set of linear extensions of $P_{n-2}$.  Therefore
the number $e(P_n)$ of linear extensions of $P_n$ satisfies
$e(P_n) = e(P_{n-1})+e(P_{n-2})$, and so $e(P_n)$ is the $n$th Fibonacci
number $F_n$ (with the convention that $F_0=F_1=1$).  Similarly, we see
that the number of linear extensions of $P_n$ with $b_1$ as the bottom
element is equal to $e(P_{n-1}) = F_{n-1}$.

Let $\nu^n$ denote the uniform measure on linear extensions of the finite
poset $P_n$.  The proportion $\nu^n(E^{P_n}(b_1))$ of linear extensions
of $P_n$ in which $b_1$ is the bottom element is equal to $F_{n-1}/F_n$,
which tends to $\phi = \frac{1}{2}(\sqrt 5 -1) = 0.618\cdots$, as
$n \to \infty$.  Similarly, for each fixed~$k$,
$$
\nu^n(E^{P_n}(b_1b_2 \cdots b_k)) = \frac{F_{n-k}}{F_n} \to \phi^k, \quad
\mbox{ as } n\to \infty.
$$
For any other ordered stem $b_{s(1)}b_{s(2)} \cdots b_{s(k)}$, where $s$
is a permutation of $[k]$ (so the set of elements in the stem is $Z_k$),
and any $n \ge k$, the linear extensions of $P_n$ with initial segment
$b_{s(1)}\cdots b_{s(k)}$ are in 1-1 correspondence with those with
initial segment $b_1\cdots b_k$, so
$\nu^n(E^{P_n}(b_{s(1)}\cdots b_{s(k)}))$ also tends to $\phi^k$ as
$n\to \infty$.

The only other $k$-element down-set of $P$ is
$W_k = \{b_1,\dots,b_{k-1},b_{k+1}\}$, and the same principle applies to
initial segments that are orderings of this set:
$\nu^n(E^{P_n}(b_1\cdots b_{k-1}b_{k+1})) = F_{n-k-1}/F_n \to \phi^{k+1}$,
and the same is true for any other ordered stem whose elements are those
of~$W_k$.

It is now natural to define
$$
\mu(E^P(a_1\cdots a_k)) = \lim_{n\to\infty} \nu^n(E^{P_n}(a_1\cdots a_k)),
$$
for each ordered stem $a_1\cdots a_k$ of $P$: we have seen that all these
limits exist, and we have found their values.  We claim that $\mu$ is an
order-invariant measure on $P$.

By Lemma~\ref{to-check}, we need to verify identities of two types:
\begin{itemize}
\item[(a)] $\mu(E^P(a_1\cdots a_k)) = \sum_c \mu(E^P(a_1\cdots a_kc))$,
for every ordered stem $a_1\dots a_k$, where the sum is over minimal
elements $c$ of $P \sm \{ a_1, \dots, a_k\}$, of which there are at most
two;
\item[(b)] $\mu(E^P(a_1\cdots a_k)) = \mu(E^P(a_{s(1)}\cdots a_{s(k)}))$,
where $s$ is a permutation of $[k]$ and both $a_1\cdots a_k$ and
$a_{s(1)}\cdots a_{s(k)}$ are ordered stems.
\end{itemize}
We could verify all these identities by direct calculation.  However, it
is just as easy to note that these identities all hold for each of the
measures $\nu^n$ with $n>k$, because the $\nu^n$ are uniform measures on
the set of linear extensions of finite posets, and therefore the
identities hold in the limit.  Here, it is crucial that the sums in (a)
are all finite sums.

On the other hand, we claim that the measure $\mu$ defined above is the
only order-invariant measure on $P$.  To prove this, it is enough to show
that $\nu(E(b_1\cdots b_k)) = \phi^k = \mu(E(b_1\cdots b_k))$ for each
$k$, for any order-invariant measure $\nu$ on $P$.  Indeed, the values of
$\nu$ for all other basic events can be derived from the values of the
$\nu(E(b_1\cdots b_k))$, assuming order-invariance, giving us that
$\nu(E(a_1\cdots a_k)) = \mu(E(a_1\cdots a_k))$ for all basic events, and
it follows that $\nu= \mu$, since the family of basic events forms a
separating class.

Let $\nu$ be an order-invariant measure on $P$, and take any $n>k$.
The set $\Xi_n$, a down-set in $P$ of size $n$, can take only the two
values $Z_n = \{b_1, \dots, b_{n-1}, b_n\}$ and
$W_n = \{b_1, \dots, b_{n-1}, b_{n+1}\}$.  We now have
\begin{eqnarray*}
\nu(E^P(b_1\cdots b_k)) & = & \nu(E^P(b_1\cdots b_k) \mid \Xi_n = Z_n)
\, \nu(\{\omega : \Xi_n(\omega) = Z_n \}) \\
&& \mbox{} + \nu(E^P(b_1\cdots b_k) \mid \Xi_n = W_n)
\, \nu(\{\omega : \Xi_n(\omega) = W_n\}). 
\end{eqnarray*}
Therefore $\nu(E^P(b_1\cdots b_k))$ lies between the two values
$\nu(E^P(b_1\cdots b_k) \mid \Xi_n = Z_n)$ and \newline
$\nu(E^P(b_1\cdots b_k) \mid \Xi_n = W_n)$.  By (\ref{uniform2}), these
two values are
$$
\nu^{Z_n}(E^{P_n}(b_1\cdots b_k)) = \nu^n(E^{P_n}(b_1\cdots b_k))
\mbox{ and } \nu^{W_n}(E^{P_{W_n}}(b_1\cdots b_k)) =
\nu^{n-1}(E^{P_{n-1}}(b_1\cdots b_k)).
$$
As both $\nu^n(E^{P_n}(b_1\cdots b_k))$ and
$\nu^{n-1}(E^{P_{n-1}}(b_1\cdots b_k))$ tend to $\phi^k$ as $n \to \infty$,
we have \newline
$\nu(E(b_1\cdots b_k)) = \phi^k$, as required.

In summary, there is exactly one order-invariant measure on $P$.

This example is considered from a slightly different perspective
in~\cite{BL1}.

\section{Extremal Order-Invariant Measures} \label{sec:extremal}

Recall that an order-invariant measure $\mu$ on $P$ is {\em extremal} if
it cannot be written as a convex combination of two different
order-invariant measures on $P$.

Two elements $\omega = x_1x_2\dots , \omega' = y_1y_2\dots $ of $\Omega^P$
are said to be {\em finite rearrangements} if for some $n \in \N$,
$\{ x_1, \dots, x_n\} = \{ y_1, \dots, y_n\}$ and, for $m > n$,
$x_m = y_m$.  A {\em tail event} in $\Omega^P$ is a subset $E$ of
$\Omega^P$ such that, if $\omega \in E$ and $\omega'$ is a finite
rearrangement of $\omega$, then $\omega' \in E$.  A measure $\mu$ is said
to have {\em trivial tail} if $\mu(E) \in \{0,1\}$ for every tail event $E$.

For $\omega = x_1x_2 \cdots \in \Omega^P$, and $k \in \N$, we can define
a measure $\nu^k(\cdot)(\omega)$ on $\Omega^P$ as the uniform measure on
the set of elements of $\Omega^P$ of the form
$x_{s(1)} \cdots x_{s(k)}x_{k+1}x_{k+2}\cdots$, where $s$ is a permutation
of $[k]$.  There are $e(P_{X_k})$ elements of this form, one corresponding
to each linear extension $x_{s(1)}\cdots x_{s(k)}$ of $P_{X_k}$.  We say
that an order-invariant measure $\mu$ on $P$ is {\em essential} if, for
every event $E \in \cF^P$, for $\mu$-almost every $\omega$,
$\nu^k(E)(\omega) \to \mu(E)$ as $k \to \infty$.

We studied the property of extremality at length in~\cite{BL1}, in the
wider context mentioned earlier.  In particular, we gave a number of
equivalent conditions for an order-invariant measure to be extremal.
These all transfer to our present setting: if an order-invariant measure
on $P$ is extremal in the space of all order-invariant measures, then it
is certainly extremal in the space of order-invariant measures on $P$;
conversely, if an order-invariant measure $\mu$ on $P$ is a convex
combination of two other order-invariant measures $\mu_1$ and $\mu_2$,
then these must both be order-invariant measures on $P$ -- meaning that,
for events $A$ such that $\mu(A)=0$ because $\mu$ is an order-invariant
measure on the fixed causal set $P$, we also have
$\mu_1(A) = \mu_2(A) = 0$ -- so if $\mu$ is extremal among
order-invariant measures on $P$, then it is extremal among all
order-invariant measures.

Putting this observation together with Theorem~7.2 and Corollary~7.4
in~\cite{BL1} gives us the following result.

\begin{theorem} \label{thm:equivalence}
Let $\mu$ be an order-invariant measure on a causal set $P$, and let $\cH$ be
a separating class in $(\Omega^P,\cF^P)$.
The following are equivalent:
\begin{itemize}
\item $\mu$ is extremal,
\item $\mu$ has trivial tails,
\item $\mu$ is essential,
\item for every event $E \in \cH$, for $\mu$-almost every $\omega$,
$\nu^k(E)(\omega) \to \mu(E)$ as $k \to \infty$.
\end{itemize}
\end{theorem}

We illustrate this result by returning to the example in the Introduction.

\bigbreak

\noindent {\bf Example 1, revisited}. \quad 
As before, let $P$ be the disjoint union of two infinite chains
$B: b_1<b_2<\cdots$ and $C:c_1<c_2<\cdots$.  For $q \in [0,1]$, let 
$\mu_q$ be the order-invariant measure on $P$ defined earlier. 

The cases $q=0$ and $q=1$ are special.  If $q=0$, then elements from $B$
are never chosen, and $\Xi=C$ a.s.; if $q=1$, then $\Xi=B$ a.s.
If $q\in (0,1)$, then $\Xi=B\cup C$ a.s.

We claim that each measure $\mu_q$ is an extremal order-invariant measure.
The easiest way to see this is to show that $\mu_q$ satisfies the final
condition in Theorem~\ref{thm:equivalence}.  Consider the event
$E(a_1\cdots a_k)$, where $a_1\cdots a_k$ is an ordered stem of $P$, and
$\{a_1,\dots,a_k\} = \{b_1,\dots,b_\ell,c_1,\dots,c_{k-\ell}\}$.  For
$\mu_q$-almost every $\omega$, we have
$|B\cap \Xi_n(\omega)|/n \to q$ as $n \to \infty$.  Now suppose that
$|B\cap \Xi_n(\omega)| = m_n(\omega) = m$; we have
$$
\nu^n(E(a_1\cdots a_k))(\omega) = \frac{\binom{n-k}{m-\ell}}{\binom{n}{m}}
= \left(\frac{m}{n}\right)^\ell \left(\frac{n-m}{n}\right)^{k-\ell}
\left( 1 - O\left(\frac{k^2}{\min(m,n-m)}\right)\right).
$$
Therefore, for any $\omega$ such that $m_n(\omega)/n$ tends to $q$, we have
\begin{equation} \label{eq:before}
\lim_{n\to \infty} \nu^n(E(a_1\cdots a_k))(\omega) =
q^\ell(1-q)^{k-\ell} = \mu_q(E(a_1\cdots a_k)).
\end{equation}
Therefore, $\mu_q$ satisfies the final condition given in
Theorem~\ref{thm:equivalence}, and hence is extremal.

Given any probability measure $\rho$ on $[0,1]$, define a probability
measure $\mu_\rho$ by first choosing a random parameter $\chi \in [0,1]$ 
according to $\rho$, then sampling according to $\mu_\chi$.  In other words,
$\mu_\rho$ is a convex combination of the order-invariant measures
$\mu_q$, so is also order-invariant.  Suppose that $\rho$ is not a.s.\
constant, so that there is some $x$ such that
$0 < p= \rho(\chi \le x) < 1$; we claim that $\mu_\rho$ is not extremal.
There are several easy arguments to show this, based on the various conditions
in Theorem~\ref{thm:equivalence}.

\begin{itemize}
\item[(a)] We can argue from the definition; for instance we can consider the
conditional probability measures $\mu^1$ and $\mu^2$ obtained by
conditioning $\mu_\rho$ on the events that $\chi \le x$ and $\chi >x$
respectively, and write $\mu_\rho = p\mu^1 + (1-p)\mu^2$.

\item[(b)] We can consider the tail event
$\limsup_{n\to \infty} |B\cap \Xi_n|/n \le x$, which has probability~$p$ not
equal to~0 or~1.

\item[(c)] We can note that $\nu^n(E(b_1))(\omega)$ a.s.\ converges to the value
$\chi$ chosen according to $\rho$, whereas
$\mu_\rho(E(b_1)) = \E_\rho(\chi)$, so $\mu_\rho$ is not essential.
\end{itemize}

The description of $\mu_\rho$ includes several apparently different
processes.  For instance, consider the following process: having chosen
the bottom $n$ elements, $m$ from $B$ and $k=n-m$ from $C$, choose the
next element to be from $B$ with probability $(m+1)/(n+2)$.  It is easy to
check directly that this defines an order-invariant process on $P$.  The
theory of {\em P\'olya's Urn} (see, for instance, Exercise~E10.1 in
Williams~\cite{williams}) tells us that the proportion of elements taken
from $B$ in the first $n$ steps converges to some limit $\chi$ as
$n\to \infty$, and that this limit $\chi$ has the uniform distribution on
$(0,1)$.  Moreover, it is possible to show that this process has the same
finite-dimensional distributions as the one defined by choosing $\chi$
from the uniform distribution in advance, then choosing the natural
extension according to $\mu_\chi$.  See Ross~\cite{ross}, Section~3.6.3.
Other urn processes correspond to other measures on $[0,1]$.

We will now show that every extremal order-invariant measure $\mu$ on $P$
is of the form $\mu_q$, for some $q \in [0,1]$.  Given such a measure
$\mu$, we set $q = \mu(E(b_1))$, the probability that the bottom element
of the natural extension is in $B$.  Our aim is to show that
$\mu(E(a_1\cdots a_k)) = \mu_q(E(a_1\cdots a_k))$ for every ordered stem
$a_1\cdots a_k$ of $P$.

For any $n \in \N$, and any $\omega \in \Omega^P$ with
$X_n = \Xi_n(\omega) = \{ b_1, \dots, b_m, c_1, \dots, c_{n-m}\}$, the
probability $\nu^n(E(b_1))(\omega)$ that the bottom element of a random
linear extension of $P_{X_n}$ is $b_1$ is equal to $m/n$, the proportion
of elements of $B$ in $X_n$.  As $\mu$ is extremal, and therefore
essential, we have that $\nu^n(E(b_1))(\omega) \to q$ a.s., and so the
proportion of elements of $B$ among the first $n$ elements also 
a.s.\ tends to $q$.

Now, take any basic event $E(a_1\cdots a_k)$, where the $a_i$ include
exactly $\ell$ elements of $B$, and any $\omega$ such that $m$ of the
first $n$ elements are in $B$.  As in (\ref{eq:before}), for any $\omega$
such that the ratio $m/n$ of elements of $B$ tends to $q$, we have
$$
\lim_{n\to \infty} \nu^n(E(a_1\cdots a_k))(\omega) = q^\ell(1-q)^{k-\ell}.
$$
We deduce that $\mu(E(a_1\cdots a_k)) = q^\ell(1-q)^{k-\ell} =
\mu_q(E(a_1\cdots a_k))$, since $\mu$ is essential.  As $\mu$ agrees with
$\mu_q$ on all basic events, $\mu$ and $\mu_q$ are equal.

Thus the $\mu_q$ are the only extremal order-invariant measures on $P$.

This example also appears in Section~2 of the paper of Kerov~\cite{Kerov},
and in~\cite{BL1}.

\bigbreak

It is not true that every extremal order-invariant measure is an extremal
order-invariant measure on some fixed $P$.  For instance, an extremal
order-invariant measure is derived from the following process: at each
step, take a label uniformly at random from $[0,1]$, and take a new
element incomparable with all existing elements.  The causal set thus
generated is a.s.\ an antichain.

As discussed at the end of Section~8 of~\cite{BL1}, every order-invariant
measure can be built from an order-invariant measure on some fixed $P$ by
a process of replacing some infinite chains of $P$ by infinite
antichains, with labels generated according to some probability
distribution on $[0,1]$.  Thus the problem of classifying extremal
order-invariant measures is reduced to the problem of classifying extremal
order-invariant measures on a fixed $P$.

Another result of~\cite{BL1} is that every order-invariant measure $\mu$
has an expression, unique up to a.s., as a {\em mixture} of extremal
order-invariant measures: there is a probability space $(W,\cG,\rho)$,
whose elements are extremal order-invariant measures $\mu_\omega$, and
$\mu$ is given by sampling $\mu_\omega$ from this space, and then sampling
from $\mu_\omega$ (more formally,
$\mu(\cdot) = \int_W \mu_\omega(\cdot)\, d\rho(\mu_\omega)$).  If $\mu$ is
an order-invariant measure on some fixed causal set $P$, then the extremal
order-invariant measures $\mu_\omega$ are, $\rho$-a.s., measures on $P$,
and so we can specify the mixture so that the $\mu_\omega$ are all
measures on $P$.

In Example~1, for instance, this implies that every order-invariant
measure on $P$ is a mixture of the $\mu_q$, that is, of the form
$\mu_\rho$ for some probability measure $\rho$ on $[0,1]$.

\section {Faithful and Non-faithful Processes} \label{sec:fnfp}

A causet process on $P=(Z,<)$, and/or its associated measure, is said to
be {\em faithful} if $\Xi(\omega)=Z$ a.s.  If a causet process is
faithful, then the associated probability measure $\mu$ is a measure on
the space $L(P)$ of natural extensions of $P$.

For instance, in Example~1 above, the measure $\mu_\rho$ is faithful
if and only if $\rho(\{0,1\}) = 0$.

If, for all elements $x$ of a causal set $P$, the set $I(x)$ of elements
incomparable to $x$ is finite, then $P$ has no proper infinite down-sets,
and therefore all causet processes are faithful.  Conversely, if $I(x)$ is
infinite for some $x$, then any causet process on the restriction
$P_{I(x)\cup D(x)}$ is also a causet process on $P$: if the restricted
process is order-invariant, then it can be seen as an unfaithful
order-invariant process on $P$.  (In Section~\ref{sec:dbt}, we shall see a
class of examples of causal sets $P$ that admit a unique order-invariant
measure, which is faithful, even though $I(x)$ is infinite for every
element $x$: there is no order-invariant causet process on any restriction
$P_{I(x) \cup D(x)}$.)

Let $\mu$ be an order-invariant measure on $P=(Z,<)$.  An element $x\in Z$
is said to be {\em absent} in $\mu$ if $x \notin \Xi$ almost surely.  Of
course, if there is an absent element in~$\mu$, then $\mu$ is unfaithful.
We shall prove that any maximal element $x$ of $P$ is absent in all
order-invariant causet processes on $P$ -- more generally, any element $x$
with no infinite chain above it is always absent.

Here and in future, when we are dealing with uniformly random linear
extensions of a {\em finite} poset, we shall denote the linear extension
$\zeta = \zeta_1\cdots \zeta_n$.

Let $P=(Z,<)$ be a finite poset.  For $x \in Z$ and $i \in [|Z|]$, we set
$r_i(x) = \nu^Z(\{ \zeta : \zeta_i = x\})$, the probability that, in a
random linear extension of $P$, $x$ is in position~$i$.

\begin{lemma} \label{lemma2}
If $x$ is a maximal element in the finite poset $P=(Z,<)$, then the
sequence $(r_i(x))$ is non-decreasing in $i$.
\end{lemma}

\begin{proof}
Set $n = |Z|$ and, for each $i=1, \dots, n$, let $L_i$ denote the set of
linear extensions $x_1\cdots x_n$ of $P$ in which $x_i = x$.  For $i<n$,
define a map $\phi_i:L_i \to L_{i+1}$ by
$$
\phi_i(x_1\cdots xx_{i+1}\cdots x_n) =
x_1\cdots x_{i+1}x\cdots x_n.
$$
This map $\phi_i$ is well-defined because, since $x$ is maximal,
$x_1\cdots x_{i+1}x \cdots x_n$ is a linear extension of $P$ whenever
$x_1\cdots xx_{i+1}\cdots x_n$ is.  For each $i$, the map $\phi_i$ is
clearly an injection, and so $|L_i| \le |L_{i+1}|$, and therefore
$r_i(x) \le r_i(x+1)$.
\end{proof}

\begin{prop} \label{maximal-absent}
Suppose $\mu$ is an order-invariant measure on a causal set $P=(Z,<)$.  If
$x \in Z$ is not absent in $\mu$, then there is an infinite chain in $P$
with bottom element $x$.

In particular, if $P$ has no infinite chain, then there is no
order-invariant measure on $P$.
\end{prop}

\begin{proof}
We start by proving that, if $x$ is maximal in $P$, then $x$ is absent
in~$\mu$.

Suppose then that $x$ is a maximal element that is not absent in $\mu$.
Now, for some $j,m \in \N$, we have
$\mu(\{\omega: \xi_j(\omega) =x\}) > 1/m$.  Set $n=m+j-1$, so that
$\mu(\{\omega: \xi_j(\omega) = x\}) > 1/(n-j+1)$.

For any stem $W$ of $P$, including $x$, with $|W|=n$, Lemma~\ref{lemma2}
tells us that $r_i^W(x) = \nu^W(\{\zeta : \zeta_i = x\})$ is
non-decreasing in $i$.  Therefore all of the $r_i^W(x)$, for
$i=j, \dots, n$, are at least $r_j^W(x)$, and so
$r_j^W(x) \le 1 /(n-j+1)$.

Let $\cW_n$ denote the set of all $n$-element stems of $P$.  For
$W \in \cW_n$, set $a_W = \mu(\{\omega : \Xi_n(\omega) =W\})$.  Thus
$\sum_{W\in \cW_n} a_W = 1$.

By order-invariance, if $x \in W$,
$$
\mu(\xi_j = x \mid \Xi_n = W) = r_j^W(x) \le \frac{1}{n-j+1},
$$
and so
$$
\mu(\{ \omega : \xi_j(\omega) = x\}) = \sum_{W: x \in W}
a_W \, \mu(\xi_j = x \mid \Xi_n = W) \le \frac{1}{n-j+1},
$$
which is a contradiction.  This proves that any maximal element $x$ is
absent in $\mu$.

To prove the full result, suppose that $\mu$ is an order-invariant measure
on $P=(Z,<)$, and let $W$ be the set of non-absent elements.  Now $\mu$ is
also an order-invariant measure on $(W,<_W)$, so this causal set has no
maximal elements.  For any element $x \in W$, we can construct an infinite
chain in $(W,<_W)$ with bottom element $x$ recursively: having found
$x=x_0<x_1< \cdots <x_k$, let $x_{k+1}$ be any element of $W$ above $x_k$.

For the final statement, if there are no infinite chains in $P$, and $\mu$
is an order-invariant measure on $P$, then every element is absent in
$\mu$, which is not possible.
\end{proof}

\bigbreak

\noindent {\bf Example 3}. \quad
Let $P=(Z,<)$ be a countably infinite antichain.  As $P$ contains no
infinite chains, there is no order-invariant causet process on $P$, by
Proposition~\ref{maximal-absent}.

In the more general context of~\cite{BL1}, there is an order-invariant
process giving rise to an antichain a.s., as discussed in that paper.
However, such a process is not an order-invariant process on a particular
labelled antichain: the labels on the elements of the generated antichain
are random.

\bigbreak

\noindent {\bf Example 4}. \quad 
Let $P$ consist of one infinite chain $b_1 < b_2 < \cdots$ together with a
single incomparable element $x$.  For any order-invariant measure $\mu$ on
$P$, the maximal element $x$ is absent in $\mu$.  Thus there is no
faithful order-invariant process on $P$, and the only order-invariant
process is the one whose measure is given by $\mu(b_1b_2\cdots) =1$; i.e.,
at each stage $i$, the process a.s.\ selects the next element $b_i$ of the
infinite chain.

\bigbreak

The causal sets in Examples~1, 3 and~4 are all upward-branching forests,
i.e., causal sets in which every element has at most one lower cover.
Equivalently, $P$ is an upward-branching forest if, for each element $x$
of $P$, the set $D[x]$ is a finite chain.  We can extend the arguments
used in the analyses of these examples as follows.

\begin{prop} \label{prop:flow}
Suppose the causal set $P=(Z,<)$ is an upward-branching forest.  Then
there is a faithful order-invariant process on $P$ if and only if $P$ has
no maximal element.
\end{prop}

\begin{proof}
If there is a maximal element $x$, then Proposition~\ref{maximal-absent}
shows that $x$ is absent, so there is no faithful order-invariant process.

If there is no maximal element, then we can define a faithful
order-invariant process via a non-zero flow $f$ through the forest, with
value~1.  To be precise, a {\em flow} in $P$ is a function
$f:Z \to {\mathbb R}^+$ satisfying $f(x) = \sum_{y\cdot> x} f(y)$ for all
$x \in Z$, where the sum is over all elements $y$ such that $(x,y)$ is a
covering pair.  The {\em value} of the flow $f$ is the sum over all
minimal elements $x$ of $f(x)$.  (To obtain a flow $g$ through the edges
(covering pairs) of the forest, in the usual sense, we set $g(x,y) = f(y)$
for each covering pair $(x,y)$.)

A flow $f(x)$ can be constructed by working recursively up the forest,
starting from the minimal elements.  The set of minimal elements is
non-empty and countable, so we can assign positive real numbers $f(x)$ to
the minimal elements summing to~1.  Once we have chosen $f(x)$, we note
that there is at least one, but only countably many, upper covers of $x$,
so we can choose positive numbers $f(y)$, for the upper covers
$y$ of $x$, so that $f(x) = \sum_{y\cdot> x} f(y)$.

Note that, given any stem $A$ in $P$, the sum of the $f(x)$ over the
minimal elements of $P\sm A$ is~1.

Given a flow $f$, our rule defining an order-invariant causet process is:
from any state $x_1\cdots x_k$, and for any minimal element $x$ of
$P\sm \{x_1, \dots, x_k\}$, the probability of a transition to the state
$x_1\cdots x_kx$ is equal to $f(x)$.

To see that a process defined in this way is order-invariant, observe
that, if $a_1\cdots a_k$ is an ordered stem of $P$, then
$\mu(E(a_1a_2\cdots a_k)) = f(a_1)f(a_2) \cdots f(a_k)$, which depends
only on the stem $\{a_1,\dots, a_k\}$, and not on the order of its
elements.
\end{proof}

One can show, using the same ideas as in Example~1, that faithful extremal
order-invariant measures on an upward-branching forest are in 1-1
correspondence with flows through the forest.

A specific example is that where $P=(Z,<)$ is a countable union
$\bigcup_{i=1}^\infty C_i$ of infinite chains.  In this case, an extremal
order-invariant measure is specified by a probability distribution on the
index set $\N$: given non-negative numbers $p_1,p_2, \cdots$ summing to~1,
an order-invariant process on $P$ is defined by the rule that, at each
step, the next element in chain $C_i$ is chosen with probability $p_i$,
independent of all other choices.  This process is faithful if all the
$p_i$ are positive.  This is an example of a faithful order-invariant
measure on a causal set $P$ containing an infinite antichain.

\medskip

To conclude this section, we discuss the case where $\mu$ is an
order-invariant measure on $P =(Z,<)$, and an element $b\in Z$ is in the
random set $\Xi$ with probability strictly between~0 and~1.  In this
situation, we can construct two new causet measures $\mu^+$ and $\mu^-$ on
$P$ by conditioning on the events $b \in \Xi$ and $b \notin \Xi$
respectively:
$$
\mu^+(E) = \mu(E \mid b \in \Xi) =
\frac{\mu(E\cap\{ \omega\in \Omega^P: b \in \Xi(\omega)\})}
{\mu(\{\omega\in \Omega^P: b \in \Xi(\omega)\})},
$$
for all $E \in \cF^P$, and similarly for $\mu^-$.  Then
$\mu(\cdot) = \mu^+(\cdot) \mu(b\in\Xi) + \mu^-(\cdot) \mu(b\notin\Xi)$, a
convex combination of $\mu^+$ and~$\mu^-$.

\begin{prop} \label{split}
If $\mu$ is an order-invariant measure on $P=(Z,<)$, and $b$ is an element
of $Z$ with $0 < \mu (\{\omega : b \in \Xi(\omega)\}) < 1$, then the
measures $\mu^+$ and $\mu^-$ defined above are order-invariant.
\end{prop}

\begin{proof}
We start by showing that $\mu^+$ is order-invariant.  Suppose that
$a_1\cdots a_k$ and $a_{s(1)}\cdots a_{s(k)}$ are two ordered stems of
$P$, where $s$ is a permutation of $[k]$: our task is to show
that
$$
\mu(E(a_1\cdots a_k) \mid b \in \Xi) =
\mu(E(a_{s(1)}\cdots a_{s(k)}) \mid b \in \Xi).
$$
Since $\mu (\{\omega : b \in \Xi(\omega)\}) > 0$, this is equivalent to
$$
\mu(E(a_1\cdots a_k) \cap \{\omega : b \in \Xi(\omega)\}) =
\mu(E(a_{s(1)}\cdots a_{s(k)}) \cap \{\omega: b \in \Xi(\omega)\}).
$$
If $b$ is one of the $a_j$, this holds directly by order-invariance.  If
not, then the set $E(a_1\cdots a_k) \cap \{\omega: b \in \Xi(\omega)\}$
can be written as a countable disjoint union of events of the form
\newline $E(a_1\cdots a_k c_1 \cdots c_t b)$.  By order-invariance, each
such event has the same probability as the corresponding event
$E(a_{s(1)}\cdots a_{s(k)} c_1 \cdots c_t b)$; summing the probabilities
now gives the required result.

We can write
$$
\mu^- (E) = \frac{\mu(E) - \mu(b\in \Xi) \mu^+(E)}{\mu(b\notin \Xi)},
$$
for every $E \in \cF^P$.  Using this identity, the fact that
$\mu(b\notin \Xi) > 0$, and the order-invariance of $\mu$ and $\mu^+$, we
see that $\mu^-$ is also order-invariant.
\end{proof}

This result is analogous to Lemma~4.3.10 of Bovier~\cite{Bovier}.

One consequence of Proposition~\ref{split} is that, if $\mu$ is an
order-invariant measure on $P=(Z,<)$, and $b$ is an element of $Z$ such
that $P\sm U[b]$ has no infinite chain, then $\mu(b \in \Xi) = 1$.
Indeed, if not, then Proposition~\ref{split} says that $\mu^-$ is an
order-invariant measure on $P\sm U[b]$, in contradiction to
Proposition~\ref{maximal-absent}.

\section {Existence of Order-Invariant Measures} \label{sec:existence}

We have seen examples where there are one, none, or many (faithful)
order-invariant measures on a fixed labelled poset $P$.  We now give a
sufficient condition for the existence of an order-invariant measure on
$P$.

\begin{theorem} \label{thm:existence}
Let $P=(Z,<)$ be a causal set.  If $P\sm A$ has finitely many minimal
elements for each stem $A$ of $P$, then there is an order-invariant
measure on $P$.  More generally, if $P_Y$ has this property for some
infinite down-set $Y$ of $P$, then there is an order-invariant measure on
$P$.
\end{theorem}

\begin{proof}
Suppose that $P\sm A$ has finitely many minimal elements for
each stem $A$ of $P$.

Let $Z_1 \subset Z_2 \subset \cdots$ be an increasing sequence of stems of
$P=(Z,<)$ whose union is $Z$.  Note that, for each ordered stem
$a_1 \cdots a_k$, $\nu^{Z_n}(E(a_1 \cdots a_k))$ is defined for all $n$
large enough that all the $a_j$ are in $Z_n$.

Since the set of all ordered stems of $P$ is countable, a standard
diagonalisation argument shows that there is a subsequence $(Z_{n_j})$ of
$(Z_n)$ such that $\lim_{j\to \infty} \nu^{Z_{n_j}}(E(a_1 \cdots a_k))$
exists for all ordered stems $a_1\cdots a_k$.

For each ordered stem $a_1\cdots a_k$, we now set
$$
\mu(E(a_1\cdots a_k)) =
\lim_{j\to\infty} \nu^{Z_{n_j}}(E(a_1 \cdots a_k));
$$
we claim that this defines an order-invariant measure on
$(\Omega^P,\cF^P)$.

For each ordered stem $a_1\cdots a_k$, the set $\{b_1,\dots,b_r\}$ of
minimal elements of $P\sm \{a_1, \dots, a_k\}$ is finite by assumption.
Provided $|Z_{n_j}| > k$, we have
$$
\sum_{i=1}^r \nu^{Z_{n_j}}(E(a_1 \cdots a_kb_i))
= \nu^{Z_{n_j}}(E(a_1 \cdots a_k)),
$$
so this identity also holds for the limit $\mu$.  (Note that
$\nu^{Z_n}(E(c_1\cdots c_t)) = 0$ unless all the $c_i$ are in $Z_n$.)
Thus, by Lemma~\ref{to-check}, $\mu$ is a causet measure on~$P$.

Checking that $\mu$ is order-invariant is also immediate: if
$a_1\cdots a_k$ is an ordered stem of $P$, and $s$ is a permutation of
$[k]$ such that $a_{s(1)}\cdots a_{s(k)}$ is also an ordered stem of $P$,
then
$$
\nu^{Z_{n_j}}(E(a_1 \cdots a_k)) =
\nu^{Z_{n_j}}(E(a_{s(1)} \cdots a_{s(k)}))
$$
for every $n_j$ for which these are defined, so this identity holds in
the limit too.

For the second statement in (1), we simply apply the first statement
to $P_Y$.
\end{proof}

\begin{corollary}
If $I(x)$ is finite for every element $x$ of $P$, then there is a
faithful order-invariant measure on $P$.
\end{corollary}

\begin{proof}
If $I(x)$ is finite for all $x \in P$, then there is certainly no
infinite antichain in $P$, and therefore the condition of
Theorem~\ref{thm:existence} is satisfied, and there is an order-invariant
measure on $P$.

Moreover, as we remarked at the beginning of Section~\ref{sec:fnfp}, a
causal set $P$ in which $I(x)$ is finite for every $x$ has no proper
infinite down-sets -- indeed, any causet process on $P$ generates each
element $x$ no later than step $|I(x) + D[x]|$ -- so all causet measures
on $P$ are faithful.
\end{proof}

If we think of two elements of $P=(Z,<)$ as ``interacting'' if they are
incomparable, then the condition that $I(x)$ is finite for every $x \in Z$
is analagous to the condition that an interaction in a spin system be
{\em regular} -- see Section~4.2 of Bovier~\cite{Bovier}, which suffices
for the existence of Gibbs measures in the context studied there (see
Corollary~4.2.17 of~\cite{Bovier}).

Example~4 illustrates these results: the poset $P$ of that example has no
infinite antichain, but there is one element $x$ with $I(x)$ infinite;
there is just one order-invariant measure on $P$, and it is not faithful.

The condition in Theorem~\ref{thm:existence} is certainly not necessary
for the existence of an order-invariant measure on $P$.  Indeed, we have
already seen examples -- see Proposition~\ref{prop:flow} and the remarks
after it -- where $P\sm A$ has infinitely many minimal elements for every
stem $A$, and yet there are infinitely many faithful extremal
order-invariant measures on $P$.

However, we do have the following result.

\begin{corollary}
Let $P=(Z,<)$ be a causal set.  Then the following are equivalent:
\begin{itemize}
\item[(1)] For every infinite down-set $Y$ of $Z$, there is an
order-invariant measure on $P_Y$.
\item[(2)] For every stem $A$ of $P$, $P\sm A$ has finitely many minimal
elements.
\end{itemize}
\end{corollary}

\begin{proof}
That (2) implies (1) follows from applying Theorem~\ref{thm:existence} to
each $P_Y$, where $Y$ is an infinite down-set of $P$.

If (2) fails, then there is a stem $A$ such that the set $M$ of minimal
elements of $P\sm A$ is infinite.  Then $A\cup M$ is an infinite down-set of
$P$ with no infinite chains, so there is no order-invariant measure on
$P_{A\cup M}$, by Proposition~\ref{maximal-absent}.
\end{proof}

It is no accident that the condition of Theorem~\ref{thm:existence} for
the existence of an order-invariant measure is the same as that in
Theorem~\ref{thm:compact} for $\Omega^P$ to be compact.  Indeed, we can
use compactness to give an alternative proof of the first part of
Theorem~\ref{thm:existence}: we merely sketch this proof, which relies
on the theory of weak compactness -- see Billingsley~\cite{billingsley}.

Since the space $(\Omega^P, \cF^P)$ is compact, every family of measures
in $\cP= \cP(\Omega^P, \cF^P)$ is tight.  Thus, by Prohorov's Theorem,
every such family, and in particular the family $\nu^{Z_n}(\cdot)$ as
defined in the proof, is relatively compact for weak convergence.
Thus some sequence of measures $\nu^{Z_n}(\cdot)$ has a weak limit: we
showed in~\cite{BL1} that a weak limit of such measures is order-invariant.

Some ``compactness'' condition is required for either proof to work.
For instance, suppose $P=(Z,<)$ is an antichain, with
$Z = \{ z_1, z_2, \dots \}$, and set $Z_n = \{z_1, \dots, z_n\}$ for each
$n \in \N$.  Now, for each fixed $k$, $\nu^{Z_n}(E(z_k)) = 1/n$ for
$n \ge k$, so $\nu^{Z_n}(E(z_k)) \to 0$ as $n \to \infty$ for each
$z_k \in Z$, although $\sum _{k=1}^\infty \nu^{Z_n}E(z_k) = 1$ for each
$n$.  A similar issue is explored in Example (4.16) in~\cite{Georgii},
where a sequence of measures tends weakly to a limit that is not a
measure on the original space: the limiting measure can be seen as a
``point mass at infinity'' in the one-point compactification of the
originally non-compact space.

\bigbreak

\section {Uniqueness of Order-Invariant Measures} \label{sec:uoim}

Our purpose in this section is to give a sufficient condition on a causal
set $P$ for $P$ to admit a unique order-invariant measure.

The following result can be seen as an interpretation of a result
from Brightwell~\cite{Bri1}.

\begin{theorem} \label{thm:unique}
Let $P=(Z,<)$ be a causal set, and suppose there is some $k$ such that
$|I(x)| \le k$ for all $x\in P$.  Then there is a unique order-invariant
measure on $P$.
\end{theorem}

\begin{proof}
For incomparable elements $a$ and $b$ of $P$, let $R(a,b)$ be the event
that $a$ appears below $b$ in a natural extension of $P$.  Formally,
$R(a,b) = \{ \omega \in \Omega^P:
\exists i<j, \xi_i(\omega)=a, \xi_j(\omega)=b\}$.

Suppose $P$ satisfies the condition of the theorem.  It is proved
in~\cite{Bri1} that, for any increasing sequence $(Z_1,Z_2,\dots)$ of
stems in $P=(Z,<)$, whose union is $Z$, and any Boolean combination $R$ of
events of the form $R(a,b)$, the limit, as $n \to \infty$, of
$\nu^{Z_n}(R)$ exists, and is independent of the choice of sequence
$(Z_n)$.

Each basic event $E(a_1\cdots a_k)$ can be written as an intersection of
events $R(a,b)$.  Also, for any $\omega = x_1x_2 \cdots \in \Omega^P$,
the union of the sequence $(X_1,X_2,\dots)$ of stems is $Z$.  Therefore,
for each ordered stem $a_1\cdots a_k$, and each $\omega \in \Omega^P$,
the result of~\cite{Bri1} tells us that $\nu^{X_n}(E(a_1\cdots a_k))$
tends to a limit, which we denote $\mu(E(a_1\cdots a_k))$, independent of
the sequence $(X_n)$.

As in the proof of Theorem~\ref{thm:existence}, this limit $\mu$ is an
order-invariant causet measure on $P$.

Moreover, for {\em every} $\omega \in \Omega^P$,
$\nu^n(E(a_1 \cdots a_k)(\omega)$ tends to $\mu(E(a_1\cdots a_k))$.
Every extremal order-invariant measure $\nu$ on $P$ is essential, by
Theorem~\ref{thm:equivalence}, and so $\nu$ must agree with $\mu$ on the
separating class consisting of the basic events $E(a_1\cdots a_k)$, and
therefore $\nu = \mu$.

Thus there is only one extremal order-invariant measure on $P$, namely
$\mu$.
\end{proof}

The condition that $I(x)$ be uniformly bounded in Theorem~\ref{thm:unique}
is reminiscent of Dobrushin's uniqueness criterion for interacting
particle systems (see~\cite{Bovier} or~\cite{Georgii}), in that it bounds
the strength of interactions.

\bigbreak

\noindent {\bf Example 5}. \quad
An example in Brightwell~\cite{Bri1} shows that just having all the
$I(x)$ finite is not sufficient to guarantee a unique order-invariant
measure.

To construct this example, we start with $P_1$ the one-element poset on
$Z_1 = \{a\}$ and $P_2$ the two-element antichain $Z_2 = \{a,b\}$.
Each $P_n$, $n\ge 3$, is constructed from $P_{n-1}$ by adding a chain of
$m_n$ elements above the elements of $Z_{n-2}$ and incomparable with the
chain $Z_{n-1}\sm Z_{n-2}$, where $m_n$ grows rapidly with $n$
($m_n = 2^{2^n}$ suffices).  The infinite poset $P$ is the union of the
$P_n$.  The point is that, as $m_n$ is much larger than $m_{n-1}$, most
linear extensions of the poset $P_n$ have the elements of $Z_{n-2}$, in
some order, as an initial segment, so $\nu^{Z_n}(E^{P_n}(a))$ can be made
as close as is desired to $\nu^{Z_{n-2}}(E^{P_{n-2}}(a))$, for each
$n\ge 3$.  Thus $\nu^{Z_{2n}}(E^{P_{2n}}(a))$ and
$\nu^{Z_{2n+1}}E^{P_{2n+1}}(a))$ tend to different limits as
$n \to \infty$.  The proof of Theorem~\ref{thm:existence} then implies
that there are at least two different order-invariant measures.  These
measures are necessarily faithful, as all the $I(x)$ are finite in this
example.

For details, see~\cite{Bri1}.

\bigbreak

On the other hand, the condition in Theorem~\ref{thm:existence} is not
necessary for the uniqueness of an order-invariant measure on a causal set
$P$.
For instance, one can build a causal set by stacking finite posets on top
of one another, with all elements of one poset in the stack being above all
elements of all posets below it.  It is easy to see that such a poset admits
a unique order-invariant measure, constructed in an obvious way from the
uniform measures on linear extensions of each poset in the stack.  This
class includes examples in which there is no uniform bound on $|I(x)|$.

\section {Downward-branching trees} \label{sec:dbt}

A {\em downward-branching forest} is a causal set in which every element
has exactly one upper cover (equivalently, for each element $x$, $U[x]$ is
a chain).  A {\em downward-branching tree}, or simply {\em tree}, is a
downward-branching forest with just one component, i.e., such that every
two elements have a common upper bound.

Our purpose in this section is to characterise the trees $T=(Z,<)$ that
admit an order-invariant measure.  Such a measure $\mu$ must be faithful:
for any element $x \in Z$, there is no infinite chain in $Z\sm U[x]$ (if
the infinite chain $U[y]$ is disjoint from $U[x]$, then $x$ and $y$ have
no common upper bound), and so, by the remark after
Proposition~\ref{split}, $\mu(x \in \Xi)=1$.

Before giving this characterisation, we state and prove two simple general
lemmas that we shall need in the course of the proof, and later.

\begin{lemma} \label{non-zero}
Let $P=(Z,<)$ be a causal set, and let $a_1a_2\cdots a_k$ be any ordered
stem of $P$.  If $\mu$ is an order-invariant measure on $P$ such that,
with positive probability, all the $a_i$ appear, then
$\mu(E(a_1a_2\cdots a_k))> 0$.

In particular, if $a$ is a minimal element of $P$, then either
$\mu(E(a))>0$, or $a$ is absent.
\end{lemma}

\begin{proof}
The event that all the $a_i$ appear is a countable union of events of the
form $E(b_1b_2 \cdots b_j)$, where all the $a_i$ appear in the set
$B=\{b_1, \dots, b_j\}$.  Thus at least one such event has positive
probability.  Now, there is a linear extension $b_{s(1)}\cdots b_{s(j)}$
of $P_B$ with initial segment $a_1\cdots a_k$.  We see that
$$
\mu(E(a_1a_2\cdots a_k)) \ge \mu(E(b_{s(1)} \cdots b_{s(j)})) =
\mu(E(b_1b_2\cdots b_j)) > 0,
$$
as required.
\end{proof}

\begin{lemma} \label{delete-stem}
Let $\mu$ be a faithful order-invariant measure on $P=(Z,<)$ and let
$A$ be any stem of $P$.  Take any linear extension $a_1\dots a_m$ of $P_A$.
For any ordered stem $b_1\cdots b_k$ of $P\sm A$, define
$$
\mu_A(E(b_1\cdots b_k)) =
\frac {\mu(E(a_1\cdots a_mb_1 \cdots b_k)}{\mu(E(a_1\cdots a_m))}.
$$
Then $\mu_A$ is a faithful order-invariant measure on $P\sm A$.
\end{lemma}

\begin{proof}
Note first that $\mu(E(a_1\cdots a_m)) > 0$, by Lemma~\ref{non-zero}, so
$\mu_A$ is well-defined.  Also, by order-invariance, it is independent
of the choice of the linear extension of $P_A$.

For any ordered stem $b_1\cdots b_k$, we need to check that the sum,
over all minimal elements $b$ of $P \sm (A \cup \{ b_1, \dots, b_k\})$,
of $\mu_A(E(b_1\cdots b_kb))$ is equal to $\mu_A(E(b_1\cdots b_k))$; this
is immediate from the definition, since $\mu$ satisfies the analogous
property.

Thus, by Lemma~\ref{to-check}, $\mu_A$ is a causet measure on $P\sm A$.
Order-invariance and faithfulness are immediate from the definition.
\end{proof}

Let $T = (Z,<)$ be a downward-branching tree.  Let $C: x_0<x_1< \cdots$ be
an arbitrary maximal chain in $T$: the minimal element $x_0$ determines
this chain $C$ uniquely as the chain $U[x_0]$ of elements above $x_0$.

For $i \ge 1$, set $B_i = D(x_i)$ and $A_i = D(x_i) \sm D[x_{i-1}]$.  Thus
$A_i$ is the finite forest of elements ``hanging off'' $C$ at $x_i$.  The
sets $A_i$ partition $T \sm C$.  Also, for each $i$,
$B_i = D(x_i) = A_i \cup D[x_{i-1}]$, and these two sets $A_i$ and
$D[x_{i-1}]$ have no comparabilities between them.
See Figure~\ref{fig:tree}.

\begin{figure} [hbtp]
\epsfxsize300pt
$$\epsfbox{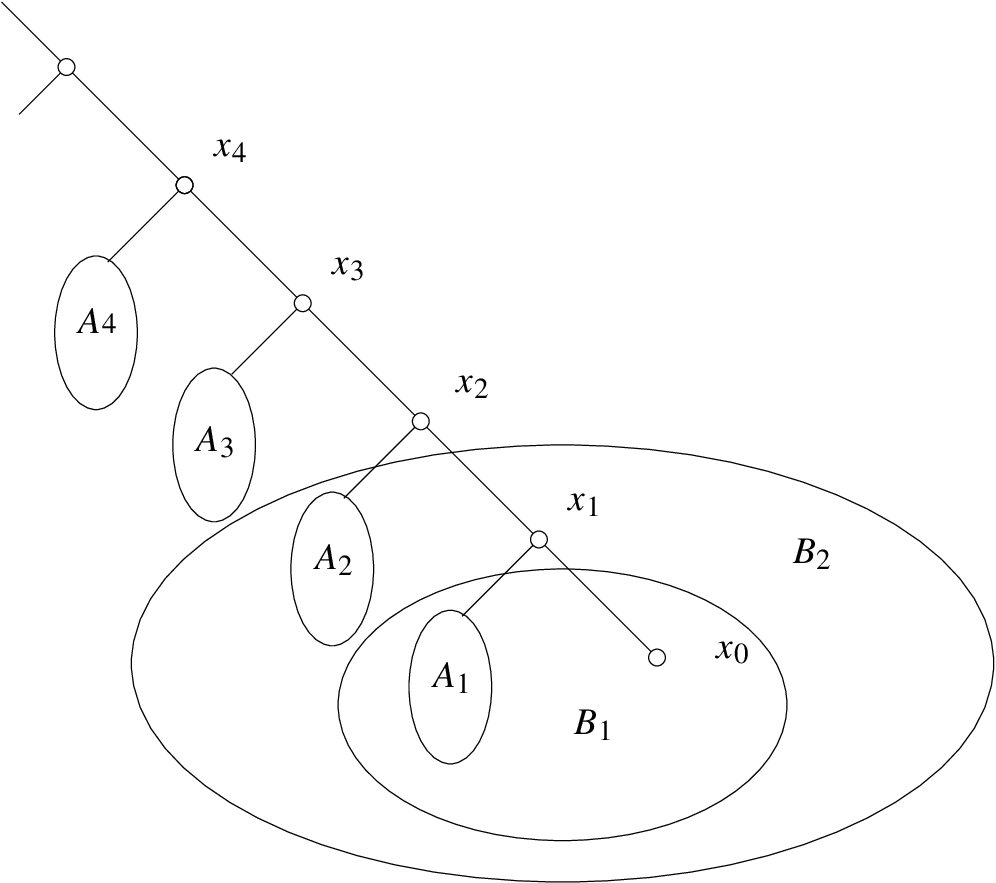}$$
\caption{A downward-branching tree} \label{fig:tree}
\end{figure}

Set $a_i = |A_i|$, $b_i = |B_i|$, and $t_i = a_i/b_i$, for each $i\ge 0$.
So $t_i$ is the proportion of elements below $x_i$ that are in subtrees
other than $D[x_{i-1}]$.

\begin{prop} \label{tree}
A tree $T=(Z,<)$ admits an order-invariant measure if and only if
$\sum_{i=0}^\infty t_i$ converges.  If the sum is convergent, there is
just one order-invariant measure on~$T$.
\end{prop}

This convergence condition is quite strong: a tree $T$ such that
$\sum t_i$ converges can be thought of as consisting of one chain $C$ with
elements hanging off it at widely spaced intervals.  For instance, if each
$a_i$ is~1, so that there is one minimal element hanging off each element
in the chain, then $b_i = 2i$ for each $i$, so $t_i = 1/2i$, and
$\sum t_i$ is divergent.  This is therefore an example of a causal set
with an infinite chain admitting no order-invariant measure.

\begin{proof}
We first note that the convergence condition is invariant under choice of
the maximal chain: given any two chains, defined by their minimal
elements, the elements have a least upper bound, which appears in both
chains, and the sequence $(t_i)$ is the same in both chains beyond this
point.

We will now show that the convergence condition is invariant under the
removal of a minimal element $x$.  Unless $T$ is a single chain -- in
which case the condition is satisfied both before and after removing the
unique minimal element $x$ -- we can choose a reference chain $C$ in
which $x$ is in one of the $A_i$.  Removing $x$ has the effect of reducing
the one term $t_i$, and increasing all subsequent terms $t_j$ by at most a
factor of~2, so the convergence of $\sum t_i$ is not affected.

We deduce moreover that the convergence condition is invariant under the
removal of any finite down-set of $T$.

Suppose that $T$ admits an order-invariant measure $\mu$, and consider the
event $E^T(x_0)$ that $x_0$ is the bottom element in a random linear
extension.  By Lemma~\ref{non-zero}, $\mu (E^T(x_0)) > 0$.

Our basic intuition is that an order-invariant measure $\mu$ on $T$, if
it exists at all, has to be the limit of the measures $\nu^{D[x_j]}$,
as $j \to \infty$.
(Indeed, if there is an order-invariant measure, then there is an extremal
one, which is essential by Theorem~\ref{thm:equivalence}, and therefore
is certainly a limit of {\em some} sequence of measures $\nu^{D_k}$, where
$(D_k)$ is an increasing sequence of down-sets of $T$.)  Accordingly, our
next step is to fix $j \ge 1$ and analyse the family of linear extensions
of $T_{D[x_j]}$, which we call $T_j$ for convenience.  As $x_j$ is the
unique maximal element of $T_j$, a linear extension of $T_j$ consists of a
linear extension of $T_{D(x_j)}$ with $x_j$ appended, so we may focus
instead on the family of linear extensions of $T_{D(x_j)}$.

In $T_{D(x_j)}$, there are no comparabilities between the sets $A_j$ and
$D[x_{j-1}]$, so a linear extension of $T_{D(x_j)}$ is determined uniquely
by: (i)~a linear extension of $T_{A_j}$, (ii)~a linear extension of
$T_{j-1}$, and (iii)~a set $I$ of $a_j$ elements of $[b_j]$.  Given
these three ingredients, the linear extension of $T_{D(x_j)}$ can be
formed by mapping the elements of $A_j$ to the elements of $I$, in the
order given by the linear extension from~(i), then mapping the elements of
$D[x_{j-1}]$ to the elements of $[b_j]\sm I$, in the order given by the
linear extension from~(ii).

The event that, in a uniformly random linear extension $\zeta$ of
$T_{D(x_j)}$, the bottom element $\zeta_1$ is in $D[x_{j-1}]$, depends
only on the set $I$, and its probability is just the probability
that $1 \not \in I$, which is $(b_j-a_j)/b_j = 1-t_j$.

Furthermore, the event that the lowest element of $D[x_{j-1}]$ is $x_0$,
in a uniformly random linear extension of $T_j$, depends only on the
linear extension of $T_{j-1}$ chosen in part~(ii) of the process
described above, so this event is independent of the event that the
overall bottom element in the linear extension of $T_j$ is in
$D[x_{j-1}]$.  Hence we have
$$
\nu^{D[x_j]}(E^{T_j}(x_0)) = (1-t_j) \nu^{D[x_{j-1}]}(E^{T_{j-1}}(x_0)),
$$
and it follows by induction that
$$
\nu^{D[x_j]}(E^{T_j}(x_0)) = \prod_{i=1}^j (1-t_i).
$$

Moreover, if $W$ is any stem including $x_j$ (and therefore all of
$D[x_j]$), then $\nu^W(E^{T_W}(x_0)) \le \prod_{i=1}^j (1-t_i)$, as the
product is the probability that $x_0$ is the lowest element of $D[x_j]$ in
a uniformly random linear extension of $T_W$.

For $j,n\in \N$, let $A_{j,n} = \{ \omega: x_j \in \Xi_n(\omega)\}$.
We have that, for $\omega \in A_{j,n}$,
$$
\nu^{\Xi_n(\omega)}(E^{T_{\Xi_n}}(x_0)) \le \prod_{i=1}^j (1-t_i).
$$
For any $j\in \N$ and $\eps > 0$, we may take $n$ sufficiently large that
$\mu(A_{j,n}) > 1-\eps$.  Now, by (\ref{uniform2}), we have that
$$
\mu(E^T(x_0)) = \sum_X \mu(E^T(x_0) \mid \Xi_n = X) \, \mu(\Xi_n = X)
= \sum_X \nu^X (E^{T_X}(x_0)) \, \mu(\Xi_n = X),
$$
where the sum is over all stems $X$ of $T$ of size $n$.  Now we have
$$
\mu(E^T(x_0)) \le \sum_{X : x_j \notin X} \mu(\Xi_n = X) +
\sum_{X: x_j \in X} \nu^X (E^{T_X}(x_0)) \, \mu(\Xi_n = X)
\le \eps + \prod_{i=1}^j (1-t_i).
$$
As both $\eps$ and $j$ are arbitrary, we conclude that
$\mu(E^T(x_0)) \le \prod_{i=1}^\infty (1-t_i)$, which is positive if and
only if $\sum t_i$ converges.

This proves that, if $T$ admits an order-invariant process, then
$\sum t_i$ converges.

Indeed, we can extract more information from the argument above.  Suppose
that $\sum t_i$ does converge.  For any minimal element $x$, decompose the
tree using the reference chain $C=U[x]$, calculate the constants
$t_i =t_i(x)$ for this chain $C$, and set
$p_T(x) = \prod_{i=1}^\infty (1-t_i(x))$.  We have seen that
$\mu(E^T(x)) \le p_T(x)$, for any order-invariant measure $\mu$ on $T$.

We claim that the sum of the $p_T(x)$ over all minimal $x$ is equal to~1.
This will imply that $\mu(E^T(x)) = p_T(x)$, for any order-invariant
measure $\mu$ on $T$, and any minimal element $x$.

Note first that, for each fixed $j$, we have
$\sum_{x\in M_j} \prod_{i=1}^j (1-t_i(x)) =1$, where the sum is over the
set $M_j$ of minimal elements of $D[x_j]$, as $\prod_{i=1}^j (1-t_i(x))$
is the probability that $x$ is the bottom element in a random linear
extension of $T_j$.

Therefore $\sum_{x\in M_j} p_T(x) = \sum_{x\in M_j}
\prod_{i=1}^\infty (1-t_i(x)) \le 1$, for each $j$.  It follows
that the sum of $p_T(x)$ over all minimal elements of $T$ is at most~1.

To see the reverse inequality, we fix any $\eps > 0$.  As $\sum t_i$
converges, there is some $n$ such that
$\prod_{i=n+1}^\infty (1-t_i) > 1 -\eps$.  Now, for all $x \in M_n$,
$t_i(x) = t_i$ for $i\ge n$.  Therefore
$$
\sum _{x \in M_n} p_T(x) = \sum _{x\in M_n}
\prod_{i=1}^n (1-t_i(x)) \prod_{i=n+1}^\infty (1-t_i) >
(1-\eps) \sum_{x\in M_n} \prod_{i=1}^n (1-t_i(x)) = 1-\eps.
$$

What this shows is that, if there is an order-invariant measure $\mu$
on $T$, then $\mu(E^T(x))$ must be equal to $p_T(x)$ for every minimal
element $x$ of $T$.

Furthermore, from any state $a_1\cdots a_k$, with
$A = \{a_1, \dots, a_k\}$, all subsequent transitions must be those of an
order-invariant process on $T\sm A$, also a downward-branching tree, by
Lemma~\ref{delete-stem}.  Therefore the probabilities for the next
transition are necessarily obtained by selecting the next minimal element
to be $x$ with probability $p_{T\sm A}(x)$.

This proves that, in the case where $\sum t_i$ converges, there is at most
one order-invariant process on $T$, namely the one described above, with
the rule that, if we have so far selected the elements of the stem $A$,
then the probability that a minimal element $x$ of $T\sm A$ is the next
element selected is $p_{T\sm A}(x)$.

It remains to show that this process is order-invariant.

The process is, by its definition, order-Markov.  We need to check that,
after the deletion of some stem $A$, $E^{T\sm A}(yz)$ and $E^{T\sm A}(zy)$
have the same probabilities, whenever $y$ and $z$ are minimal elements
of $T\sm A$.  Without loss of generality, $A = \emptyset$ and $y=x_0$.  We
choose $n$ so that $z < x_n$.

We see that
$$
\mu(E^T(yz)) = p_T(y) p_{T\sm \{y\}}(z) =
\prod_{i=1}^\infty (1-t_i(y)) \prod_{i=1}^\infty (1-t'_i(z)),
$$
where the $t_i$ are calculated in $T$, and the $t'_i$ in $T \sm \{y\}$.
Similarly
$$
\mu(E^T(zy))= \prod_{i=1}^\infty (1-t_i(z))
\prod_{i=1}^\infty (1-t''_i(y)),
$$
where the $t''_i$ are calculated in $T\sm \{z\}$.  In each product, all the
terms beyond the $n$th are identical, so we need to prove that
$$
\prod_{i=1}^n (1-t_i(y)) \prod_{i=1}^n (1-t'_i(z)) =
\prod_{i=1}^n (1-t_i(z)) \prod_{i=1}^n (1-t''_i(y)).
$$
But these products are exactly $\nu^{D[x_n]}(E^{T_n}(yz))$ and
$\nu^{D[x_n]}(E^{T_n}(zy))$ respectively, so they are indeed equal.
\end{proof}

One explicit way of realising the unique order-invariant measure in the
case when $\sum t_i$ converges is as follows.  Again, we need only
describe how to generate the first element.  Choose a reference chain $C$
with minimal element $x_0$, and define the $A_i$ with respect to $C$ as
before.  Mark each set $A_i$ with probability $t_i$, independently of
other marks.  Note that no empty $A_i$ is marked, and, by the
Borel-Cantelli Lemma, since $\sum t_i$ is finite, there are a.s.\ only a
finite number of marked $A_i$.  If there are any marked sets, let $A_k$ be
the last marked set, take a uniformly random linear extension of the
finite poset $A_k$, and select the bottom element of this linear extension
as our first element.  If there are no marked sets, we choose $x_0$ as our
first element.  We omit the detailed analysis.

\section{The Two-Dimensional Grid Poset} \label{sec:grid}

Let $G=(\N \times \N,<)$ be the infinite two-dimensional grid poset, with
$(a,b) \le (c,d)$ if $a\le c$ and $b\le d$.  This is a causal set, with
unique minimal element $(1,1)$.

This example is studied in detail in papers of Gnedin and Kerov~\cite{GK},
Kerov~\cite{Kerov} and Vershik and Tsilevich~\cite{VT}.  Our account will
be a sketch only.

As $G$ has no infinite antichain, Theorem~\ref{thm:existence} tells us
that there is an order-invariant measure on $G$ -- however, this is
actually trivial in this case, as the chain $H=\{ (a,1):a\in \N\}$ forms
an infinite down-set in $G$, and the process that always selects the next
element of $H$ is certainly order-invariant.

In fact, there is a faithful order-invariant measure on $G$.  Although
$I(x)$ is infinite for all elements $x$ of $G$ except the unique minimum,
the method used in the proof of Theorem~\ref{thm:existence} can be used
directly to construct such a measure.  If we take $Z_n = [n]\times [n]$, a
down-set in $G$, for each $n \in \N$, then the numbers of linear
extensions of subposets of $G_{Z_n}$ can be calculated using the hook
formula of Frame, Robinson and Thrall~\cite{FRT}, and so it is possible to
write down an expression for $\nu^{Z_n}(E(a_1\cdots a_k))$ for each $n$
and any ordered stem $a_1\cdots a_k$.  It turns out that
$\nu^{Z_n}(E(a_1\cdots a_k))$ converges to a positive limit for each
ordered stem $a_1\cdots a_k$, and so the limit is a faithful
order-invariant measure on $G$.  This measure is the well-known
{\em Plancherel measure} (see, for instance, \cite{AD,VK}).

However, this is far from the only faithful order-invariant measure on
$G$.  For example, for $\alpha \in (0,1)$, we construct an order-invariant
measure as follows.  We decompose $G$ as the union of the chain
$H=\{ (a,1):a\in \N\}$, and $G\sm H$, which is isomorphic to $G$.  On the
poset formed as the disjoint union $H \cup (G\sm H)$, where the relations
between $H$ and $G\sm H$ are deleted, we construct a process which, at
each step, takes the next element of the chain $H$ with probability
$\alpha$, and otherwise takes an element from $G\sm H$ according to the
Plancherel measure.  With positive probability, the sequence constructed
is actually a natural extension of $G$: conditioning on this event gives
an order-invariant measure on $G$.  The order-invariant measure we obtain
``favours the first row $H$'', as elements of this row are chosen a
positive proportion of the time in the process, unlike in the Plancherel
measure.

It is easy to see that this can be extended, to obtain processes favouring
more than one row, and/or favouring the low-numbered columns.
Kerov~\cite{Kerov} shows that the extremal order-invariant measures on $G$
are in 1-1 correspondence with pairs of sequences
$\alpha_1 \ge \alpha_2 \ge \cdots \ge 0$,
$\beta_1 \ge \beta_2 \ge \cdots \ge 0$, such that
$\sum_{i=1}^\infty \alpha_i + \sum_{i=1}^\infty \beta_i \le 1$.  (The
measure described above is the one corresponding to $\alpha_1=\alpha$,
with all other $\alpha_i$ and $\beta_i$ equal to zero.)

\section {Open Problems} \label{sec:op}

We finish by mentioning a number of open problems.

\medskip

\noindent
{\bf 1.} Is there some reasonably simple description of the class of
causal sets that admit a (faithful) order-invariant measure?  We see from
consideration of the class of downward-branching trees that there can be
no {\em very} simple description.  However, perhaps Theorem~\ref{tree} may
give some indication of the nature of a possible classification of causal
sets admitting an order-invariant measure.

\medskip

\noindent
{\bf 2.} Is there some reasonably simple description of the class of
causal sets that admit a {\em unique} order-invariant measure?  This seems
likely to be harder than the previous problem.

\medskip

In~\cite{BL1}, we give a description of the general form of any
extremal order-invariant measure on the space $(\Omega,\cF)$.  In order to
extend this to a classification of extremal order-invariant measures, it
would suffice to be able to describe all extremal order-invariant measures
on fixed causal sets.  It is not clear what such a description might look
like, but solving Problems~{\bf 1} and~{\bf 2} would be progress towards
this goal.

\medskip

\noindent
{\bf 3.} One specific problem relates to a partial order obtained by
taking a Poisson process $X$ in ${\mathbb R}^2_+$, and taking the
partial order $<$ on $X$ induced by the co-ordinate order.  The poset
$P=(X,<)$ is a.s.\ a causal set.  Does such a poset (a.s.) admit an
order-invariant measure?

\medskip

If so, it seems that $P$ will (a.s.) admit infinitely many
order-invariant measures, because an order-invariant measure $\mu$ must
have $\sum_{x\in M} \mu(E(x)) > 1-\eps$ for some finite set $M=M(\eps)$ of
minimal elements $x$, whereas the Poisson process itself has no
distinguishable ``minimum region'' of finite area.  This is because the
Lebesgue measure on ${\mathbb R}^2_+$, and hence the Poisson process, are
Lorentz invariant (i.e., invariant under the measure-preserving
transformations $(x,y) \to (ax,a^{-1}y)$ of ${\mathbb R}^2_+$).

The motivation behind this problem comes from physics.  Any process
generating a random causal set can be viewed as a potential discrete model
for the space-time universe.  Rideout and Sorkin~\cite{RS,RS2} proposed
(essentially) order-invariance as a desirable feature of such a model.
It would be good to know whether the (rich) class of order-invariant
processes does include processes that produce outcomes resembling the
observed space-time universe, i.e., at least locally resembling a Poisson
process in 4-dimensional Minkowski space $M^4$.  If such a process exists,
it will have an expression as a mixture of extremal order-invariant
processes on fixed causal sets, where the causal sets ``resemble'' those
produced from a Poisson process.

It seems likely that either (i)~causal sets arising from a Poisson process
in $M^4$ (with an origin) a.s.\ admit an order-invariant measure, or
(ii)~there is some necessary structural condition for the existence of an
order-invariant measure that is not satisfied by any causal set
``faithfully embedded'' into $M^4$.  It would be very interesting to know
which.

The 2-dimensional version of this question, as proposed above, should be
easier to settle.

\medskip

\noindent
{\bf 4.} We give a more specific question, an answer to which is likely to
lead to an answer to Problem~{\bf 3}.  Let $P = (X,<)$ denote the causal set
defined from a Poisson process in the positive quadrant, as above.  For each
$n$, consider the restriction $P_n$ of $P$ to the set $X_n$ of points in the
square $[0,n]^2$.

Now let $x=(u,v)$ be the point in $X$ with minimum sum of co-ordinates
$u+v$.  Consider the probability $q_n = \nu^{X_n}(E^{P_n}(x))$ that $x$ is
the bottom element of a uniform random linear extension of $P_n$.  Does
$q_n$ a.s.\ tend to zero as $n\to \infty$?

\medskip

If $q_n$ does (a.s.) tend to zero, then it should be fairly easy to deduce
that, in any order-invariant measure $\mu$ on $P$, $\mu(E^P(x)) = 0$, and
thence that there is no order-invariant measure on $P$.

On the other hand, if $q_n$ tends to some non-zero limit, and also
$\nu^{X_n}(E^{P_n}(y))$ converges for every other minimal element~$y$, with
the sum of these limits being~1, then it seems very likely that the
measures $\nu^{X_n}$ will have a limit that is an order-invariant measure
on $P$.

The following version of the question seems likely to be equivalent, and
may be slightly more appealing.  If we generate $P_n$ as above, and then
take a random linear extension of $P_n$, does the probability that the
bottom element lies in $[0,1]^2$ tend to zero as $n \to \infty$?

\medskip

\noindent {\bf 5.}
Can one say anything interesting about the causet properties ``$P$ admits
an order-invariant measure'' and ``$P$ admits a unique order-invariant
measure''.  Could one or other be monotone (i.e., preserved under adding
relations)?  The following example shows that the property ``$P$ admits a
faithful order-invariant measure'' is not monotone.

\bigbreak

\noindent {\bf Example 6}. \quad
Let $P=(Z,<)$ consist of two chains $B: b_1<b_2< \cdots$ and
$C: c_0<c_1<c_2<\cdots$, with also the `cross-relations' $c_i > b_j$ if
$j < 2^i$ -- so each element $c_i$ has $2^i-1$ elements of $B$ below it.
This causal set $P$ is obtained from the one in Example~1, which does
admit a faithful order-invariant measure, by adding relations.  We shall
show that, in any order-invariant measure $\mu$ on $P$, $c_0$, and hence
all the elements of $C$, are absent.

For $n \ge 1$, let $X$ be any down-set of $P$ of size $2^n$ containing
$c_0$.  Thus $c_n \notin X$, and so $X$ contains at most $n$ elements of
$C$.  If $Q=(Y,<')$ is the poset with $Y = X$ consisting of the union of
the two chains $C\cap X$ and $B\cap X$, without the cross-relations, then
$\nu^Y(E^Q(c_0))= |C\cap X|/|X| \le n/2^n$.  Now the theorem of Graham,
Yao and Yao~\cite{GYY} implies that adding the cross-relations (which
means conditioning on certain events that the $c_j$ are higher than the
$b_i$) cannot increase the probability that $c_0$ is below $b_1$: thus
$$
\nu^X (E^{P_X}(c_0))\le \nu^Y (E^Q(c_0)) \le \frac{n}{2^n}.
$$
As in the proof of Proposition~\ref{maximal-absent}, this implies
that, in any order-invariant measure $\mu$ on $P$,
$\mu(E^P(c_0)) \le n/2^n$ for every $n$, so $\mu(E^P(c_0)) =0$.  Finally,
by Lemma~\ref{non-zero}, we see that $c_0$, and hence all the $c_i$, are
absent in $\mu$.

\bigbreak

The property of admitting an order-invariant measure is preserved under
the addition of {\em finitely many} relations to $P$: conditioning an
order-invariant measure $\mu$ on the event that a linear extension of $P$
respects those extra relations yields an order-invariant measure on the
causal set with the relations added.

However, we do not know whether the property of admitting an
order-invariant measure is preserved under the {\em removal} of finitely
many relations.

\end{document}